\documentclass[english,12pt]{amsart}
\usepackage{amsmath}
\usepackage{amssymb}
\usepackage{amsthm}
\usepackage[T1]{fontenc}
\usepackage{color}
\usepackage{textcomp}
\usepackage[]{enumerate}
\usepackage{pst-node}
\usepackage{tikz-cd}
\usepackage[english]{babel}

\newtheorem{theorem}{Theorem}
\newtheorem{definition}[theorem]{Definition}
\newtheorem{lemma}[theorem]{Lemma}
\newtheorem{proposition}[theorem]{Proposition}
\newtheorem{corollary}[theorem]{Corollary}

\newtheorem{remark}[theorem]{Remark}
\newtheorem{example}[theorem]{Example}

\newcommand{\defeq}{\mathrel{\mathop:}=}

\newcommand{\N}{\mathbb{N}}
\newcommand{\e}{\varepsilon}
\newcommand{\de}{\delta}

\newcommand{\C}{\mathcal C[0,1]}
\newcommand{\G}{\mathcal G}
\renewcommand{\S}{\mathbb S}
\renewcommand{\O}{\mathcal O}
\newcommand{\J}{\mathcal J}

\makeatletter
\newcommand{\subalign}[1]{%
  \vcenter{%
    \Let@ \restore@math@cr \default@tag
    \baselineskip\fontdimen10 \scriptfont\tw@
    \advance\baselineskip\fontdimen12 \scriptfont\tw@
    \lineskip\thr@@\fontdimen8 \scriptfont\thr@@
    \lineskiplimit\lineskip
    \ialign{\hfil$\m@th\scriptscriptstyle##$&$\m@th\scriptscriptstyle{}##$\crcr
      #1\crcr
    }%
  }
}
\makeatother

\begin{document}
\title[pseudo-physical measures]{pseudo-physical measures for typical continuous maps of the interval}

\author{Eleonora Catsigeras}

\author{Serge Troubetzkoy}

\address{Instituto de Matem\'atica y Estad\'istica ``Prof.\ Ing.\  Rafael Laguardia'' (IMERL), Universidad de la Rep\'ublica, Av.\ Julio Herrera y Reissig 565, C.P. 11300, Montevideo, Uruguay}
\email{eleonora@fing.edu.uy}
\urladdr{http:/fing.edu.uy/{\lower.7ex\hbox{\~{}}}eleonora}

\address{Aix Marseille Univ, CNRS, Centrale Marseille, I2M, Marseille, France}
\address{postal address: I2M, Luminy, Case 907, F-13288 Marseille Cedex 9, France}
\email{serge.troubetzkoy@univ-amu.fr}
\urladdr{http://iml.univ-mrs.fr/{\lower.7ex\hbox{\~{}}}troubetz/} \date{}

\thanks{We gratefully acknowledge support of the projects Physeco, MATH-AMSud,   "Sistemas Din\'{a}micos" funded by CSIC of  Universidad de la Rep\'{u}blica (Uruguay),   and  APEX "Syst\`{e}mes dynamiques: Probabilit\'{e}s et Approximation Diophantienne PAD" funded by the R\'{e}gion PACA.
The project leading to this publication has also received funding from Agencia Nacional de Investigaci\'{o}n e Innovaci\'{o}n (ANII) of Uruguay, Excellence Initiative of Aix-Marseille University - A*MIDEX and Excellence Laboratory Archimedes LabEx (ANR-11-LABX-0033), French ``Investissements d'Avenir'' programmes.}

\begin{abstract}
We study the measure theoretic properties of typical $C^0$ maps of the interval.
We prove that any ergodic measure is pseudo-physical, and conversely, any  pseudo-physical measure is in  the
closure of the ergodic measures, as well as in the closure of the atomic measures.    We show  that the set of pseudo-physical measures is meager in the space of
all invariant measures.
Finally, we study the entropy function. We construct pseudo-physical measures with infinite entropy. We also prove that, for each $m \ge 1$, there exists infinitely many pseudo-physical measures with entropy $\log m$, and deduce that the entropy function is  neither upper semi-continuous nor lower semi-continuous.

\end{abstract}\maketitle

\section{Introduction}\label{s1}

The topological dynamics of $C^0$-generic maps of compact manifolds has been intensively studied by Hurley and his co-authors  \cite{H1,H2,AHK}.
The measure theoretic properties of
$C^0$-generic volume preserving maps of compact manifolds have also been well studied, starting with the famous Ulam-Oxtoby theorem: \emph{the generic volume preserving maps of compact
manifolds is ergodic \cite{OU}}, and continued by extensive work by Alpern and Prasad \cite{AP}.  Outside the volume preserving maps,
the measure theoretic properties of  $C^0$-generic systems on compact manifolds have first been studied by Abdenur and Andersson \cite{AA}.  Their main result
is that a $C^0$-generic map has no physical (i.e., SRB) measure, however  the Birkhoff average of any continuous function is convergent
for Lebesgue a.e.\ point.
This result is the starting point of our investigation.  We specialize to the case of the interval, and describe the pseudo-physical properties  of the ergodic measures for
$C^0$-generic maps.

\subsection{Definitions}
Throughout the article we consider the space $\C$ of continuous maps of the interval $[0,1]$ to itself.  For
  $I,J \subset [0,1]$ let $d(I,J) \defeq \inf \{ |x - y| : x \in I, j \in J\}$.


We will often refer to a Borel probability measures on $[0,1]$  simply as a  \em measure. \em
We  consider the following distance between measures,  which induces the weak$^*$ topology on the set of all measures
\begin{equation}
\label{eqnDistanceOfMeasures}
\mbox{dist} (\mu, \nu) \defeq \sum_{i = 1}^\infty \frac{1}{2^i} \left |    \int \Psi_i \, d\mu - \int \Psi_i \, d\nu \right |\end{equation}
where $\{\Psi_i\}_{i \in \N}$ is a countable dense subset of $\C$.

For any point $x \in [0,1]$,  let  $p\omega(x)$  be the set of the Borel probability measures on $[0,1]$
that are the limits in the weak$^*$ topology of the convergent
subsequences of the following sequence
$$\left \{  \frac1n \sum_{j=0}^{n-1} \delta_{f^jx} \right\}_{n \in \N}$$
where $\delta_y$ is the Dirac delta probability measure supported in $y \in [0,1]$.

\begin{definition} \em A measure $\mu$  is called \em physical or SRB \em if $p\omega(x) = \mu$ for a set $A(\mu)$ of positive Lebesgue measure. Note that we do not require physical measures to be ergodic.
\end{definition}

In this article we consider the following generalization of the above definition, introduced in \cite{CE1} and studied in the $C^1$ case in \cite{CE2}, \cite{CCE1}, \cite{CCE2}:

\begin{definition} \em \label{definitionpseudo-physical}
A probability measure $\mu$ is called \em pseudo-physical for $f$ \em if for all $\e > 0$ the set $A_\e(\mu) \defeq \{x \in [0,1]: \mbox{dist} (p\omega(x),\mu) <  \nolinebreak
\e\}$ has positive Lebesgue
measure.  We denote by $\O_f$   the set of  pseudo-physical measures for $f$.

(In \cite{CE1,CE2,CCE1} pseudo-physical measures were called SRB-like.)

Note that any pseudo-physical measure is automatically $f$-invariant, and   we do not require a pseudo-physical measure to be ergodic.
\end{definition}

\begin{definition}
\label{definitionTypicalMaps} \em
We will call a property  \emph{typical}  if there is a dense $G_{\delta}$ subset $G \subset \C$ such that all $f \in G$ share this property.
\end{definition}

\subsection{Statement  of  results}
The main results of this article are that the following properties are typical for a map   $f \in \C$:

\noindent$\bullet$ The map $f $ has non countably many pseudo-physical measures supported on fixed points; hence with zero entropy (Theorem \ref{theoremSRB-l-withZeroEntropy}).

\noindent$\bullet$ Any atomic invariant measure supported on a periodic orbit is pseudo-physical, and the set of pseudo-physical measures is the closure of such atomic  measures with zero entropy (Theorem \ref{TheoremSRB-l=ClosureErgodic}).

\noindent$\bullet$ Any ergodic measure is pseudo-physical, and the set of pseudo-physical measures is the closure of the ergodic measures (Theorem \ref{TheoremSRB-l=ClosureErgodic}).

\noindent$\bullet$ The set of pseudo-physical measures is meager in the space of invariant measures; moreover, it is a closed set with empty interior (Theorem \ref{theoremSRB-l_Meager}).

\noindent $\bullet$ For any natural number $m \geq 2$, there exists infinitely many pseudo-physical ergodic measures whose metric entropy is exactly $\log m$ (Theorem \ref{theorempseudo-physicalMeasEntropyLog_m}).

\noindent $\bullet$ There exists pseudo-physical ergodic measures whose metric entropy is infinite (Theorem \ref{theorempseudo-physicalInfiniteEntropy}).

\noindent $\bullet$ The map $f$ is non expansive (Corollary \ref{corollaryNonExpansive}). Moreover, $f$ is Lebesgue-almost everywhere strongly non-expansive (Definition \ref{definitionStronglyNonExpansive}, Corollary \ref{corollaryStronglyNonExpansive}).

\noindent $\bullet$ The entropy function is neither upper semi-continuous nor lower semi-continuous. Moreover, it is also neither upper nor lower semi-continuous  when restricted to the set of pseudo-physical measures. (Corollary \ref{CorollaryNoSemiContinuity}).

\section{A technical lemma}\label{s2}
We will use the following lemma  on the approximation of measures.

\begin{lemma}\label{lemma-dist}
For any   $\e >0 $, there exists  $q \geq 1$ such that,
if $\mu$ and $\nu$ are  probability measures satisfying:
$supp (\nu) \cup supp (\mu) \subset \cup_{j=1}^m I_j$ for some pair-wise disjoint closed intervals with nonempty interiors,  lengths at most $1/q$ and
$|\nu(I_j) - \mu(I_j) | \le 1/qm$ for each $j$
then $\mbox{dist}(\mu, \nu) < \e$.
\end{lemma}

\begin{proof}
 Fix $n \geq 1$ such that $\sum_{i= n+1} ^{+ \infty} 2^{-i} < {\e}/{2}.$ Then, for any pair of probability measures $\mu, \nu$ we obtain
 \begin{equation}
 \label{eqn01}
 \mbox{dist} (\mu, \nu) < \frac{\e}{2} +  \sum_{i= 1}^n   \frac{1}{2^i} \left |    \int \Psi_i \, d\mu - \int \Psi_i \, d\nu \right |.
 \end{equation}
 From the uniform continuity of the finite family of functions $\{\Psi_i\}_{1 \leq i \leq n}$, there exists $\delta > 0$ such that, if  $|x_1-x_2| < \delta$, then $|\Psi_i(x_1) - \Psi_i(x_2)|< \e/4$ for all $1 \leq i \leq n$.
Fix $q  \in \mathbb{N}^+$  such that
  $ {1}/{q} <  \min (\delta, \e/4) $.

The mean value theorem for integrals, yields for all $1 \leq i \leq n$
\begin{equation*}
  \begin{split}
   \int\Psi_i \, d\mu &= \sum _{j= 1} ^m \int_{I_j} \Psi_i  \, d \mu = \sum _{j= 1} ^m\Psi_i(x_j) \mu (I_j) \mbox{ for some } x_j \in I_j,\\
    \int\Psi_i \, d\nu & = \sum _{j= 1} ^m \int_{I_j} \Psi_i  \, d \nu = \sum _{j= 1} ^m\Psi_i(x'_j) \nu (I_j) \mbox{ for some } x'_j \in I_j.
  \end{split}
  \end{equation*}
\nopagebreak[4]
From this we deduce
\begin{equation*}
\begin{split}
 \int\Psi_i \, d\mu - \int \Psi_i \, d\nu  &  = \sum _{j= 1} ^m \left (\Psi_i(x_j) \mu (I_j) -\Psi_i(x'_j) \nu (I_j) \right ) \hspace{2in} \phantom{a} \\
\hspace{1.1in}  & =\sum _{j= 1} ^m  \left (\Psi_i(x_j) (\mu (I_j) -  \nu (I_j) ) +  (\Psi_i(x_j)   -\Psi_i(x'_j) ) \nu (I_j) \right ).
\end{split}
\end{equation*}
Since  $|\nu(I_j) - \mu(I_j) | \le 1/qm$ for each $j$, we conclude
$$ \left |\int\Psi_i \, d\mu - \int \Psi_i \, d\nu \right | \leq \sum _{j= 1} ^m  \frac{1}{qm}    +  \frac{\e}{4} \sum _{j= 1} ^m   \nu (I_j) < \e/2 \ \ \ \forall \ 1 \leq i \leq n.   $$
Substituting this inequality in (\ref{eqn01}),  finishes the proof of Lemma \ref{lemma-dist}.
\end{proof}

 \section{Abundance of zero-entropy pseudo-physical measures} \label{sectionSRB-l-zeroEntropy}
\begin{theorem} \label{theoremSRB-l-withZeroEntropy}
For a typical  map in $\C$ there exists an uncountable set $\mathcal F$ of  fixed points such that for each $x \in \mathcal F$ the associated Dirac delta measure $\delta_x$  is  a pseudo-physical measure.
\end{theorem}

As a consequence, there exists uncountably  many pseudo-physical measures with zero entropy.

\begin{proof}
Any  map $f \in \C$  has at least one fixed point.
Typically, the extremes of the interval are not fixed points.

For any $q \ge 1$ let
 $C^0_{0,q}$ be the set of $f \in \C$ for which there
exists   $\delta $ satisfying 0$ < \de  <  1/q$ such that $f$	is constant in an open interval of length  $\de /2$ centered at some fixed point $x_0$, and $f$ has at least $q$
distinct fixed points
$\{x_i\}_{i=1}^q$ different from $x_0$ in an open  interval of length $ \delta$ centered at $x_0$, and furthermore for each $x_i$ the map $f$ is constant on a small interval containing it.
Denote by $C^0_q$ the set of  $f \in \C$ such that, for some fixed point $x_0$, there exists at least $q$ other fixed points $\{x_i\}_{i=1}^q$   in an open interval of length  $0 <\delta < 1/q$ that contains $x_0$,  and for each  $i = 0,\dots, q$ there exists a closed interval $I^i_q$ of length
$0 < \delta_i  < \eta$ such that
$x_i \in int(I^i_q)$, the image $f(I^i_q)$ is contained in the interior of $I^i_q$ where $\eta \defeq 3 \cdot \min d(x_i,x_j)$. This choice
of $\eta$ guarantees that the $I^i_q$ are disjoint, and the gap in between any pair of them is at least of size $\eta$.

Then, $C^0_{0,q}$ is dense and   $C^0_{q}$ is open. Since $C^0_{0,q} \subset C^0_{q}$, we deduce that $C^0_{q}$ is open and dense. Consider
the dense $G_\de$ set
$$\G \defeq \bigcap_{q \ge 1} C^0_q.$$
For any $f \in \G$, for each $q \ge 1$
we have  $f \in C^0_{q}$ with associated fixed points $\{x_i(q)\}_{i=0}^{q}$, associated intervals
$\{I^i_{q}\}_{i=0}^{q}$ and gap size $\eta_q > 0$. Note that $d(I^i_q,I^j_q) \ge \eta_q$ for each $i \ne j$.

Consider the sequences $\{\vec i \defeq (i_1,i_2,\dots): 0 \le i_q \le q\}$ for which the points $x_{i_q}(q)$ converge to a point $x_{\vec i}$.
There are uncountably many such sequences with distinct limit points $x_{\vec i}$, each one is a fixed point of the map $f$.  The sequence of intervals
$I^{i_q}_q$ satisfy the following
\begin{align}\label{close}
d_q \defeq d(x_{\vec i},I^{i_q}_q)  & \le d(x_{\vec i},x_{i_q}(q)) + d(x_{i_q}(q),I^{i_q}_q) \\ \nonumber
  & =  d(x_{\vec i},x_{i_q}(q))  \to 0 \text{ as } q \to \infty. \nonumber
\end{align}

Fix $\vec i$ and $q$. Consider the map $g \defeq g_{i_q} \defeq f|_{I^{i_q}_q}: I^{i_q}_q \mapsto I^{i_q}_{q}$ and a pseudo-physical measure $\mu_{i_q}$ of $g$.

Let $J^{i_q}_q \defeq \{x \in [0,1]: d(x,I^{i_q}_q) \}  < 2 d_q\}$, note that $x_{\vec i} \in J^{i_q}_q$. Pseudo-physical measures are invariant, in particular the
the measures $\mu_{i_q}$ are  invariant. Fix $\e  > 0$ and apply
Lemma \ref{lemma-dist} for sufficiently large $q$ to the measure $\mu_{i_q}$  and $\delta_{x_{\vec i}}$  (here $m=1$, $\mu_{i_q}(J^{i_q}_q )
=  \delta_{x_{\vec i}}(J^{i_q}_q) = 1)$ to conclude that
$\mbox{dist}(\mu_{i_q},\delta_{x_{\vec i}} ) < \e$ for all sufficiently large $q$, and thus
$$\lim_{q \to \infty} \mu_{i_q} = \delta_{x_{\vec i}}. $$
Since the compact support of  $\mu_{i_q}$ is contained in $I^{i_q}_q$, this measure
 is also pseudo-physical for $f$.
Besides the set of pseudo-physical measures for $f$ is closed, we conclude that $\delta_{x_{\vec i}}$ is pseudo-physical for $f$.
\end{proof}

An $f \in \C$ is called \em expansive in the future \em  if there exists a   constant $\alpha >0$,  called the  \em expansivity constant, \em such that, for any two points $x, y \in [0,1]$,
  if  $d(f^n(x), f^n(y) ) \leq \alpha $ for all $  n \in \mathbb{N},$ then $  x=y.$
\begin{corollary} \label{corollaryNonExpansive}
A typical  map in $\C$ is not expansive in the future.
\end{corollary}
\begin{proof}
In the proof of Theorem \ref{theoremSRB-l-withZeroEntropy},  we constructed arbitrarily small  intervals $I_q$ of length $0 < \delta < 1/q$ such that $f(I_q) \subset I_q$. Thus, for any pair of points $x \neq y$ in $I_q$, the distance of their images is less than $1/q$.  Hence, if $f$ were expansive, the expansivity constant must be smaller than $1/q$ for all $q \in \N^+$, which is a contradiction.
\end{proof}

Thus, we can not apply  the classical theorem  of upper semi-continuity of the entropy function for expansive maps to typical maps, to bound from below the entropy of a measure by the entropy of nearby measures. In fact,  in the sequel we will prove that, for typical $f \in \C$, the entropy function is neither upper  nor lower semi-continuous.

\section{Shrinking intervals cover Lebesegue a.e.}

A  \em  periodic shrinking interval of period $p \geq 1$,  \em is a nonempty open interval $I$ such that  $\{f^j( \overline I)\}_{0 \leq j \leq p-1}$ is a family of pairwise disjoint sets,  $f^p(\overline I) \subset I$, and   $\mbox{length} (f^j(I) ) < \mbox{length}(I)$ for all $1 \leq i \leq p-1$.

An nonempty open interval $J$  is \em eventually periodic shrinking,  \em if there exists a periodic shrinking interval  $I$ and an $n \in \N^+ $ such that  $ f^n(J) \subset I$, and $\mbox{length} (I), \mbox{length}(f^j(J)) <  \mbox{length}(J)$ for all $1 \leq j \leq n-1$ .

A periodic or eventually periodic shrinking interval   for some $f \in \C$  is also a periodic or eventually periodic shrinking interval  for all $g$ in a small neighborhood of $f$ in $\C$. Any periodic shrinking interval $I$ of period $p$ contains at least one periodic point of period $p$.
If $I$ is a periodic or eventually periodic shrinking interval with period $p$, then $\mbox{length}(f^j(I)) < \mbox{length}(I)$ for all $j \geq 1$.

\begin{theorem} \label{theoremLebesgue-a-e-isInShrinkingInterval}
For a typical map  Lebesgue-a.e.\ point $x \in [0,1]$ belongs to a sequence  $\{I_q\}_{q \in \mathbb{N}+}$ of  arbitrarily small eventually periodic shrinking intervals $I_q$; and  the length of the $I_q$ goes to zero as $q \rightarrow + \infty$.
\end{theorem}
\begin{proof}
For   $q,k \in \N^+$, denote by ${\mathcal S}_{q,k}$  the set of maps in $\C$, for which there exists a finite family of nonempty open intervals,
which we denote by $\{  I_ {1}, \ldots,   I_h\}$,     such that:

\noindent (i)  $Leb(I_i) < 1/q$ for   $i= 1, 2, \ldots, h$,

\noindent (ii) $I_i$ is a periodic or eventually periodic shrinking interval,

\noindent (iii) $Leb ([0,1]\setminus \bigcup_{i= 1}^k I_i) < 1/k$.

The set ${\mathcal S}_{q,k}$ is open   in $\C$. In fact, for each $f \in {\mathcal S}_{q,k}$, the same family of shrinking intervals of $f$, is also a family of shrinking intervals for any other map $g$ in a sufficiently small neighborhood of $f$.
Now, let us prove that ${\mathcal S}_{q,k}$ is dense in $\C$.
Fix any map $f \in \C$. For any   $\e >0$,
 let us construct  a perturbation $g$ of $  f$ in ${\mathcal C}$,  and a family $\{I_i\}_ {1 \leq i \leq h}$ of shrinking intervals that  satisfies conditions (i) (ii) and (iii).
Take $0 <\delta < \e/3$ such that $d(f(x_1), f(x_2)) < \e/2$ whenever  $d(x_1, x_2) < \delta$.
Choose $N   \in \mathbb{N}$ such that $N > q$ and $ 1/N < \delta$. Consider the partition of $[0,1]$ into exactly $N$ intervals $J_1, \ldots, J_N$, all with the same length $1/N$.
Let $x_i$ be the midpoint of $J_i$. By construction $f(J_i) \subset (f(x_i)- \e/2, f(x_i) + \e/2)$. Since $(1/N) < \delta < \e/3 $, there exists at least one interval $J_{j(i)} \subset (f(x_i)- \e/2, f(x_i) + \e/2)$. Consider   the midpoint $x_{j(i)}$ of $J_{j(i)}$ and construct $\widehat g: [0,1]\mapsto [0,1]$ by the equality $\widehat g|_{J_i} \defeq x_{j(i)} $ for all $i= 1, \ldots, N$. The function $\widehat g$ is piecewise constant, possibly discontinuous at the boundary points of each interval $J_i$. Any point in $[0,1]$ is eventually periodic by $\widehat g$, and by construction
\begin{equation}
\label{eqn02} d(\widehat g(x), f(x)) < \e \ \ \forall \ x \in [0,1].
\end{equation}
Besides, if $\partial J_i \cap  \partial J_{l} \neq \emptyset$, then the distance between the midpoints $x_i$ and $x_l$ of the intervals $J_i$ and $J_l$ respectively is exactly $1/N$. Thus $f(x_l) \in (f(x_i)- \e/2, f(x_i) + \e/2) $, and thus  $d (x_{j(i)}, x_{j(l)}) < 2 \e$. Therefore   the discontinuity jumps $\Delta_i$ of $\widehat g$ are all smaller than $2 \e$ in absolute value.

For each $i$ consider an open sub-interval  $I_i \subset J_i$ such that the left boundary point of $I_i$ coincides with the left boundary point of $J_i$, and such that
$$0 < \mbox{length}(J_i \setminus I_i) <\frac{1}{Nk}. $$  Define the continuous piecewise affine map $  g: [0,1]\mapsto [0,1]$ such that,
$  g|_{I_i} = g|_{I_i}$  and    $g|_{J_i \setminus I_i}$ is affine for all  $i= 1, \ldots, N$. Since the absolute values   of the discontinuity jumps  of $  g$ are  smaller than $2 \e$, we deduce that
$d(  g(x), \widehat g(x)) < 2 \e \ \ \forall \ x \in [0,1].$
Together with  (\ref{eqn02}) this implies
\begin{equation}
\label{eqn03}
  \mbox{dist}(g , f ) < 3 \e.\end{equation}
  By construction, the family $\{I_i\}_{1 \leq i \leq N}$ of open intervals satisfies properties (i), (ii) and (iii) for $ g \in {\C}$. Thus $g \in {\mathcal S}_{q,k}$.
 We have proved that  ${\mathcal S}_{q,k}$ is dense in $\C$.

To end the proof of Theorem \ref{theoremLebesgue-a-e-isInShrinkingInterval}, consider the set
$${\mathcal S} \defeq \bigcap_{q \in \mathbb{N}^+}\bigcap_{k \in \mathbb{N}^+} {\mathcal S}_{q,k}.$$
On the one hand, since ${\mathcal S}_{q,k}$ is open and dense in $\C$, we have that typical maps of $\C$ belong to ${\mathcal S}$. By construction of ${\mathcal S}_{q,k}$, any map in ${\mathcal S}$ satisfies the conclusion of the  Theorem.
\end{proof}
\begin{definition} \em \label{definitionStronglyNonExpansive}
A map $f \in \C$ is \em Lebesgue-a.e.~strongly non expansive in the future, \em  if for any real number $\alpha >0$,   and for Lebesgue a.e.~$x \in [0,1]$,
$$Leb\Big (\{ y \in [0,1]: d(f^n(x), f^n(x))< \alpha \ \  \forall \ n \in  \mathbb{N}  \}   \Big) > 0.$$

\end{definition}

\begin{corollary} \label{corollaryStronglyNonExpansive}
A typical  map in $\C$  is Lebesgue-a.e.\ strongly non expansive in the future.
\end{corollary}

\begin{proof}
Take any periodic or eventually periodic shrinking   interval $I$  with length smaller than $\alpha$. Then,  $\mbox{length}(f^j(\overline I)) <\alpha$ for all $j \geq 0$. Any two points $x, y \in I$  satisfy $\mbox{dist}(f^j(x), f^j(y)) < \alpha $ for all $j \geq 0$. Thus for any point $x \in I$ we have
$$ Leb\Big (\{ y \in [0,1]:  d(f^n(x), f^n(x))< \alpha \ \  \forall \ n \in  \mathbb{N}  \}   \Big) \geq Leb(I) > 0.$$
Thus the result follows directly from  Theorem \ref{theoremLebesgue-a-e-isInShrinkingInterval}.
\end{proof}

\section{Approximation of pseudo-physical measures by periodic orbits.}

\begin{lemma}
\label{lemmaInvMeasOnShrinkingIntervals} Let $f \in \C$ and $\mu$ be an $f$-invariant measure. Assume that $I \subset [0,1]$ is a periodic shrinking interval of period $p$, and denote  $K \defeq  \bigcup_{j= 0}^{p-1} f^j(\overline I).$ Then $$\mu (f^j(\overline I)) = \mu(\overline I) = \frac{\mu(K)}{p}   \ \ \ \forall \ j  \geq 0. $$
\end{lemma}
\begin{proof}
 For any $j \geq 0$ let $j= kp - i$  with $k \in \N^+$ and $0 \leq i <p$. Using that $\mu$  is $f$-invariant, that $f^p(\overline I) \subset I$ and that $f^{-j} (f^j (\overline I)) \supset \overline I$ for all $j \geq 0$, we obtain:
 $$\mu (f^j(\overline I)) =  \mu (f^{-j}(f^j(\overline I)) \geq \mu (\overline I) \geq \mu(f^{kp}(\overline I)) = \mu (f^{-i} (f^{kp}(\overline I)) \geq \mu(f^j(\overline I)).$$
Hence, all the inequalities above are equalities; and thus $\mu(f^j(\overline I)) = \mu(\overline I)$ for all $j \geq 0$.
The result follows since $\mu(K) = \sum_{j= 0}^{p-1} \mu ( f^j(\overline I) )$.
\end{proof}
Let ${\mathcal M}_ f$ be the set of \emph{$f$-invariant Borel probability measures} and
 $\mbox{Per}_f$ be  the set of all the \emph{atomic periodic  $f$-invariant measures}:
 $$\{\mu \in {\mathcal M}_f: \mu = \frac{1}{p} \sum_{j= 0}^{p-1} \delta_{f^j(x_0)}
\mbox{ where } x_0 \mbox{ is a periodic point of period } p\}.$$
We denote by $\mathcal E_f$ the set of \emph{ergodic invariant measures} for $f$. In particular, we have $\mbox{Per}_{f} \subset {\mathcal E}_f$
 \begin{theorem} \label{TheoremPeriodicSRB-l}
 For a typical map  $f \in \C$,
 ${\mathcal O}_ f \subset \overline{\mbox{Per}_{f}} \subset \overline {{\mathcal E}_f}.$

 \end{theorem}
\begin{proof}
Choose  $q_n \geq 1$ such that  any two  measures  satisfying the $q_n$-approach conditions of  Lemma \ref{lemma-dist}, are mutually at distance smaller than $1/n$.
Let $\mu \in \O_f$.
By the definition of pseudo-physical measure, the set $A_n \subset [0,1]$ of points $y$ such that \begin{equation}
\label{eqn09}
\mbox{dist}(p \omega (y), \mu)< 1/n, \end{equation} has positive Lebesgue measure. If $f$ is $\C$ typical, Lebesgue a.e.~$y_n \in A_n$ is contained in an arbitrarily small eventually periodic shrinking interval (Theorem \ref{theoremLebesgue-a-e-isInShrinkingInterval}). 
Therefore, there exists a  point $y_n \in A_n$ such that   any measure    $\mu_n \in p \omega({y_n})$ is supported on $K_n \defeq \bigcup_{j= 1}^{p_n} f^j(\overline I_n)$,  where $I_n$ is a periodic shrinking interval  of period $p_n$,  such that $\mbox{length}(I_n) \leq 1/q_n$. From the definition of shrinking interval,   and applying Lemma \ref{lemmaInvMeasOnShrinkingIntervals}, we have $\mu_n(f^j(\overline I)) = 1/p_n$.

Since the interval $I_n$ is shrinking periodic with period $p_n$, there exists at least one periodic orbit of period $p_n$ in $K_n$. Consider the atomic periodic invariant measure $\nu_n $ supported on this periodic orbit. It also satisfies $\nu_n(f^j(\overline I)) = 1/p_n$. So, applying Lemma \ref{lemma-dist} we deduce that
\begin{equation}
\label{eqn10} \forall \  y \in A_n,  \ \forall \  \mu_n \in p\omega(y), \  \exists \   \nu_n \in \mbox{Per}_{f}\colon \
 \mbox{dist}(\mu_n, \nu_n) \leq 1/n. \end{equation}

Combining Inequalities (\ref{eqn09}) and (\ref{eqn10}), we deduce that there exists    sequence of periodic atomic measures $\nu_n$ such that
 $\mbox{dist}(\nu_n, \mu) < 2/n$, finishing the proof of Theorem \ref{TheoremPeriodicSRB-l}.
\end{proof}

\section{Most invariant measures are not pseudo-physical.}
 The main purpose of this section is to  prove that for a fixed typical map  $f \in \C$, the set $\O_f$ of   pseudo-physical measures  is topologically meager in the set ${\mathcal M}_f$ of  all the $f$-invariant measures. Precisely, its complement ${\mathcal M}_f \setminus \O_f$ is open and dense in ${\mathcal M}_f$.

The following is an immediate corollary of  Theorem 3.6 and Proposition 4.1 of Abdenur and Anderson in \cite{AA}:

\begin{theorem}
 {\bf (Corollary of Abdenur-Anderson Theorem) } \label{theoremAA}

\noindent Let $f \in \C$ be a   $C^0 $- typical map. Then

\noindent {a) } for Lebesgue almost every point $x \in [0,1]$, the  sequence $$\Big \{\frac{1} {n}  \sum_{j=0}^{n-1} \delta_{f^j(x)}\Big\} _{n \in \mathbb{N}^+}$$ of empirical probabilities is convergent,

\noindent {b) } there do not exist SRB-measures, and

 \noindent {c) } the limit set in the space of probabilities of the sequence $$ \Big \{\frac{1} {n}  \sum_{j=0}^{n-1} {{f^*}^j(\mbox{Leb})}\Big\} _{n \in \mathbb{N}^+}$$ is strictly included in the set of $f$-invariant measures.
\end{theorem}

\begin{proof}
In \cite[Theorem 3.6 and Proposition 4.1]{AA}   the above assertions a) b) and c) are proved for typical $C^0$ maps $g: M\mapsto M $, where $M$ is a manifold without boundary. In particular they hold  for a typical continuous map  $g$ of the circle $\S^1$.
Now, let us prove that they also hold for a typical continuous map $f$ of the interval. For any continuous map  $f \in \C$, we  consider $g: \S^1 \mapsto \S^1$ as follows.
First, take $\S^1 = [-1,1]/\sim$ where the equivalence relation $\sim$ identifies the extremes $-1$ and $1$.
Consider any $g: \S^1 \to \S^1$ such that  $g|_{[0,1]}= f$, $g(-1) = f(1)$.

An open set of $f \in \C$ yields an open set of $g \in \mathcal{C}(\S^1)$, thus if
 $g$ is typical in $\mathcal{C}(\S^1)$  then $f$ is typical in $\C$. Applying properties (a), (b) and (c) to   $g$, we deduce that $f$ also satisfies them.
\end{proof}

For any  $x \in [0,1]$ we define
\begin{equation}
\label{eqn-mu_x}
\mu_x \defeq \lim_{n \to \infty} \frac{1} {n}  \sum_{j=0}^{n-1} \delta_{f^j(x)}\end{equation} when this limit exists.
Let
\begin{equation*}
AA \defeq \{  x \in [0,1]: p\omega(x) = \{\mu_x\}\} \mbox{ and }
AA_1 \defeq \{  x \in AA: \mu_x \in \O_f\}.\end{equation*}

\begin{remark}
\label{remarkAA&AA_1} \em
For typical maps in $\C$,  Theorem \ref{theoremAA} a) tells us $$Leb(AA) = 1,$$ while Proposition 3.1 and Definition 3.2 of \cite{CE2} imply that $$Leb(AA_1) = 1.$$ \end{remark}

\begin{proposition} \label{propositionAA_1}
If $f $ is typical  in $\C$, then
$$\O_f = \overline{\{ \mu_x: x \in AA_1 \}},$$
and $\{\mu_x: x \in AA_1\}$ has no isolated points. Moreover, for any $\e > 0$ and any  $x_0 \in AA_1$ the set  $\{x \in AA_1: \mu_x \neq \mu_{x_0} \mbox{ and } \mbox{dist}(\mu_x, \mu_{x_0}) < \e\}$ has positive Lebesgue measure.
\end{proposition}

\begin{proof}  By definition of the set $AA_1$ we have
$\{\mu_x: x \in AA_1\}  \subset \O_f$. Besides, since $\O_f$ is closed in the weak$^*$ topology (Theorem 1.3 in \cite{CE1}) we have
 $\overline{\{\mu_x: x \in AA_1\}}  \subset \O_f$.

 We turn to the other inclusion.  Suppose $\mu \in \O_f$, i.e.\
 $$Leb\{ x \in [0,1]: \mbox{dist}(p\omega(x),\mu) < \e) > 0\} \text{ for all } \e > 0.$$
Since $Leb(AA_1) =1$ and $p\omega(x) = \mu_x$ for $x \in AA_1$,
this is equivalent to
  $$Leb\{ x \in AA_1: \mbox{dist}(\mu_x,\mu) < \e) > 0\} \text{ for all } \e > 0.$$
Thus $\mu \in \overline{\{ \mu_x: x \in AA_1 \}}$. Since $\mu \in \O_f$ is arbitrary we conclude
$$\O_f \subset \overline{\{ \mu_x: x \in AA_1 \}}.$$

Now, let us prove the last assertion of Proposition \ref{propositionAA_1}. 
By contradiction, assume that there exists $x_0 \in AA_1$ and $\e >0$ such that \begin{equation}
\label{eqn08}
\mbox{Leb}\{x \in AA_1\colon \mu_x \neq \mu_{x_0} \mbox{ and } \mbox{dist}(\mu_{x_0}, \mu_x) < \e\} = 0.\end{equation}
Since $\mu_{x_0}$ is pseudo-physical,  by definition we have
$$\mbox{Leb}(\{x \in [0,1]\colon \mbox{dist}(p\omega x, \mu_{x_0}) < \e\}) >0.$$
From Remark \ref{remarkAA&AA_1}  we deduce that
$$\mbox{Leb}(\{x \in AA_1\colon \mbox{dist}(\mu_x, \mu_{x_0}) < \e\}) >0.$$
Combining this last assertion with Inequality (\ref{eqn08}), we obtain that
$$\mbox{Leb}(\{x \in AA_1\colon  \mu_x = \mu_{x_0}) \}) >0.$$
So, $\mu_{x_0}$ is an SRB measure, contradicting part b) of Theorem \ref{theoremAA}.
\end{proof}

\begin{lemma} \label{Lemma-ConvexCombination}
Suppose $f \in \C$ is a typical map,
 $\mu_1$ is an $f$-invariant measure supported on  $K = \bigcup_{j= 1}^{p} f^j(\overline I)  $, where $  I$ is a periodic shrinking interval of period $p $, and if $\mu_2$ is another $f$-invariant measure whose compact support is contained in the complement of $K$, then no convex combination $ \nu= \lambda \mu_1 + (1 - \lambda) \mu_2$, with $0 < \lambda < 1$, is pseudo-physical.
\end{lemma}
\begin{proof}
 Since $I$ is shrinking periodic with period $p$, it is an open interval and   $f^p(\overline I) \subset I $. Construct a continuous function $\psi$ such that $\psi|_{f^p(\overline I)}= 1$, $0 < \psi(x) <1$ if $x  \in I \setminus f^p(\overline I) $ and $\psi(x) = 0$ if $x \not \in I$. Since $\mu_1 $ is supported on $K =   \bigcup_{j= 0}^{p-1} f^j(\overline I) $,  we can apply Lemma \ref{lemmaInvMeasOnShrinkingIntervals} to deduce that $$\mu_1(\overline I) = \mu_1 (f^p(\overline I)) = \frac{1}{p}; \ \ \mbox{ hence }  \mu_1 (I \setminus f^p(\overline I)) = 0.  $$ We obtain $
  \int \psi \, d \mu_1 = \mu_1 (f^p(\overline I)) =  {1}/{p} $ and $ \int \psi \, d \mu_2 = \mu_2 (I) = 0.$
  Therefore
 $ 0 <\int \psi \, d \nu =  \lambda /p  <  {1}/{p}.  $

Choose $\e >0$  such that for any measure $\mu$, if $\mbox{dist}(\nu, \mu) < \e$, then $\int \psi \, d \mu > 0$ and $\int \psi \, d \mu <   1 /p. $
Consider the set $$A_{\e}(\nu) \defeq \{ x \in AA: \mbox{ dist} (\mu_x, \nu) < \e\}.$$
We claim that  the set $A_{\e} (\nu)$ is empty. Arguing by contradiction, assume that there exists $x \in A_{\e}(\nu)$.
Then,  from the choice of $\e$, we have   \begin{equation}
   \label{eqn07}0 <\int \psi \, d \mu_x < 1/p, \end{equation} and  from (\ref{eqn-mu_x}) we deduce that there exists $n_0 \geq 1$ such that $$ \frac{1}{n} \sum_{j= 0} ^{n-1} \psi (f^j(x)) = \int \psi \, \Big (\frac{1}{n} \sum_{j= 0}^{n-1} \delta_{f^j(x)}
   \Big )  >0 \ \forall \ n \geq n_0.$$
Thus,  there exists $n_1 \geq 1$ such that $\psi(f^{n_1}(x))    >0$; hence $f^{n_1}(x) \in I$. This implies that  the future orbit of $f^{n_1}{x}$ is contained in $K$. Hence,  $\mu_x$ is supported on $K$. Since $\mu_x$ is $f$-invariant, we apply Lemma \ref{lemmaInvMeasOnShrinkingIntervals} to deduce that $\mu_x (f^p(\overline I)) = 1/p$ and $\mu_x(I \setminus f^p(\overline I)) = 0$. Therefore,
$$\int \psi \, d \mu_x = \mu_x(f^p(\overline I)) = \frac{1}{p},$$
contradicting  (\ref{eqn07}). We have proved that $A_{\e} (\nu) $ is   empty.

For a typical map in $\C$,
 the definition of $\mu$ being pseudo-physical  and   Remark \ref{remarkAA&AA_1}, imply that for all $\e >0$ we have   $$\mbox{Leb}(A_{\e}(\mu)) >0. $$   Since in our case $A_{\e}(\nu) = \emptyset$, we conclude that $\nu$ is not a pseudo-physical measure if $f$ is a typical map in $\C$.
\end{proof}

In a Baire-space, we say that a set $A$ is \em topologically meager  \em if it is the countable union of nonempty closed sets with empty interiors; and
 $A$ is \em less than topologically meager \em if it is closed with empty interior.
\begin{theorem}
\label{theoremSRB-l_Meager}
If $f$ is typical in $ \C$, then the set  $\O_f$  of pseudo-physical measures is less than topologically meager.
\end{theorem}
\begin{proof}
From \cite[Theorem 1.3]{CE1} the set $\O_f$ is closed. Now, let us prove that  its interior  in ${\mathcal M}_f$ is empty. Fix $\mu \in \O_f$ and fix $\e >0$. From Proposition \ref{propositionAA_1}, we can find $x_0 \in AA_1$ such that $\mbox{dist} (\mu_{x_0},\mu) < \e/2$, and there exists a positive Lebesgue measure set of $x \in AA_1$ such that $\mbox{dist}(\mu_x, \mu_{x_0}) < \e/2$, and thus by the triangle inequality
the set    \begin{equation}
\label{eqn12}
A'_\e(\mu) \defeq \{x \in AA_1:   \mbox{dist}(\mu_x, \mu) < \e\} \end{equation}
has positive Lebesgue measure.

Applying Theorem \ref{theoremLebesgue-a-e-isInShrinkingInterval}, there exists $x \in A'_\e(\mu)$   belonging to an arbitrary small eventually periodic shrinking interval. Hence $\mu_x$ is supported in the $f$-orbits of arbitrarily small periodic shrinking intervals. Applying Proposition \ref{propositionAA_1}, there exists a positive Lebesgue measure set of  $y \in A'_\e(\mu)$  with $\mu_y \neq \mu_x$, and such that $\mu_y$ is also supported on the $f$-orbits of arbitrary small periodic shrinking intervals. Since $\mu_x \neq \mu_y$ we claim that there   exist two periodic shrinking intervals $I_x$ and $I_y$ whose $f$-orbits support $\mu_x$ and $\mu_y$ respectively, and are mutually disjoint. If not, Lemmas \ref{lemmaInvMeasOnShrinkingIntervals} and \ref{lemma-dist}, would imply that the distance between $\mu_x$ and $\mu_y$  is arbitrarily small. Hence both measures would coincide, contradicting our choice of them.

Thus, we can apply Lemma \ref{Lemma-ConvexCombination} to deduce that the convex combination $\nu = \lambda \mu_1 + \lambda \mu_2$, with $0 < \lambda <1$, is non pseudo-physical.  But $\mbox{dist} (\nu, \mu) \leq \lambda \mbox{dist}(\mu_1, \mu) +(1 - \lambda) \mbox{dist}(\mu_2 , \mu) < \e$. We have proved that any pseudo-physical measure is accumulated by non pseudo-physical measures. In other words, the interior of the set $\O_f$ is empty.
\end{proof}

\section{pseudo-physical measures and ergodicity.}

A \em  $\delta$-pseudo-orbit of $f$ \em is a sequence $\{y_n\}_{n \in  \mathbb{N}} \subset [0,1]$ such that $$\mbox{dist}(f(y_n), y_{n+1} ) < \delta \ \ \forall \ n \in \mathbb{N}.$$
A \em $\delta$-pseudo-orbit  $\{y_n\}_{n \in  \mathbb{N}} $ is  periodic \em with period $p \geq 1$, if $$y_{n+p}= y_n \ \ \ \forall \ n \in \mathbb{N}.$$
A  map $f \in \C$ has the   \emph{periodic shadowing property} if for all $\e > 0$, there exists $\delta>0$ such that,  if $\{y_n\}_{n \in \mathbb{N}}$ is any periodic $\delta$-pseudo-orbit, then, at least one  periodic orbit $\{f^n(x)\}_{n \in \mathbb{N}}$ satisfies
$$d(f^n(x), y_n) < \e \ \ \ \forall \ n \in \mathbb{N}.$$

\begin{theorem}
\label{TheoremOprocha&als} {\bf (Ko\'{s}cielniak--Mazur--Oprocha--Pilarczyk)}

\noindent A typical map $f \in \C$ has the   periodic shadowing property.
\end{theorem}
\begin{proof}
See \cite[Theorem 1.2]{Oprocha&als}
\end{proof}

Recall that ${\mathcal E}_f$ denotes the set of ergodic measures, and $\mbox{Per}_{f}$ denotes the set of atomic invariant measures supported on periodic orbts of $f$.
As a consequence of Theorem \ref{TheoremOprocha&als}:
\begin{corollary}
\label{corollaryOprocha&als}
For a typical map $f \in \C$,   $\mathcal E_f \subset \overline{\mbox{Per}_{f}}. $
\end{corollary}
\begin{proof}
It is standard to check that the definition distance in the weak$^*$ topology of the space of probability measures implies that   for all $\e_0 >0$, there exists $\e>0$, such that, for any two points $x_1, x_2 \in [0,1]$, $$d(x_1, x_2) < \e \ \ \Rightarrow \ \ \mbox{dist}(\delta_{x_1}, \delta_{x_2})< \e_0.  $$

Fix any $\mu  \in {\mathcal E}_f$. Since $\mu$ is  ergodic, we have
$p \omega (x) = \{\mu\}  \ \ \mbox{ for } \mu\mbox{-a.e. } x \in [0,1].$ Fix such a point $x$, then there exists $n_0 \geq 1$ such that
\begin{equation}
\label{eqn11}
\mbox{dist} \left (\frac{1}{n} \sum_{j= 0}^{n-1} \delta_{f^j(x)}, \ \mu \right ) < \e_0 \ \ \ \forall \ n \geq n_0.\end{equation}
Given $\e$, choose $\delta >0$ given by Theorem \ref{TheoremOprocha&als}. Being well distributed, the point $x$ must be recurrent, thus
$ d(f^{p - 1} (x), x) < \delta \ \ \mbox{ for some } p \geq n_0$.
Construct the periodic $\delta$-pseudo-orbit  $\{y_n\}_{n \in \mathbb{N}}$ of period $p$ defined by   $y_n= f^n(x)$ for all $0 \leq n < p$,
$y_{n + p} = y_n$ for all $n \geq 0$.  Applying Theorem \ref{TheoremOprocha&als}, there exists a periodic orbit $\{f^n(z)\}_{n \geq 0}$, such that
$d (f^n(z), y_n) < \e \ \ \forall \ n \geq 0$.
By construction,  if  $ip \le  n < (i+1)p$    and   $i \ge 0$ then
$d(f^n(z), f^{n-ip}(x)) < \e $
and thus, from the choice of $\e$, we obtain
$\mbox{dist} (\delta_{f^n(z)}, \delta_{f^{n-ip}(x)}) < \e_0$
Denote by $q$ the period of $z$.
Taking into account that  balls are convex in the
 weak$^*$-distance in the space of probabilities, we deduce
$$\mbox{dist}\left ( \frac{1}{qp} \sum_{j= 0 }^{qp-1} \delta_{f^j(z)},  \ \ \frac{1}{qp}  q \cdot \sum_{j= 0 }^{p-1} \delta_{f^j(x)}   \right) < \e_0. $$
For the atomic invariant measure $\nu$ supported on the periodic orbit of $z$, we have
$$\nu = \frac{1}{q} \sum_{j= 0}^{q-1} \delta_{f^j(z)} =  \frac{1}{qp} \sum_{j= 0 }^{qp-1} \delta_{f^j(z)}.$$
Thus,
$$\mbox{dist}\left ( \nu ,  \ \ \frac{1}{p} \sum_{j= 0 }^{p-1} \delta_{f^j(x)}   \right) < \e_0.$$
Together with  (\ref{eqn11}), this implies that 
the given ergodic measure $\mu$ is $2  \e_0$-approximated by  some measure $\nu \in \mbox{Per}_{f}$, but $\e_0 > 0$ was arbitrary.
\end{proof}

 An invariant measure $\mu$ is called \em infinitely shrinked \em if there exists a sequence $\{I_q\}_{q \geq 0}$ of periodic shrinking intervals $I_q$, of periods $p_q$, such that $\mbox{length}(I_q) < 1/q$ and $\mu(\bigcup_{i= 1}^{p_q}f^j(\overline I_q) )= 1$ for all $q \geq 1$.
We denote by $\mbox{Shr}_f \subset {\mathcal M}_f$ the set of  infinitely shrinked  invariant measures.
Define
$$AA_2  \defeq \{ x \in AA_1: \ \  \mu_x \in \mbox{Shr}_f\}. $$
\begin{theorem}
\label{theoremSRB-l=ClosureSrh}
For a typical map $f \in \C$,  $$ \mbox{ Leb}(AA_2)  = 1 \ \ \mbox{ and  } \ \  \O_f  =\overline {\{\mu_x \colon  x \in AA_2\}} =  \overline {\mbox{\em Shr}_f}.$$
\end{theorem}

\begin{proof}
From Theorem \ref{theoremLebesgue-a-e-isInShrinkingInterval}, Lebesgue-a.e.~$x \in [0,1]$ belongs to a sequence of eventually periodic or periodic shrinking intervals $J_q$ with $\mbox{length}(J_q) < 1/q$. Every eventually periodic shrinking interval $J_q$ is a wandering under  $f$ until it drops into a periodic shrinking interval $I_q$ with
$ \mbox{length}(I_q)< \mbox{length}(J_q)< 1/q$.  By the definition of periodic shrinking interval, every point of $I_q$ has all the measures of $p\omega (x) $ supported on the compact set $$K_p\defeq \bigcup_{j= 1}^{p_q} f^j(\overline I_q).$$ In particular for Lebesgue almost all $x \in AA_1$, the limit measure $\mu_x$  defined by   (\ref{eqn-mu_x}), is supported on $K_p$. Thus, for a.e.\ $x \in AA_1$ we have $\mu_x \in \mbox{Shr}_f$. We have proved that the set $AA_2$ has full Lebesgue measure.

By construction,  $AA_2 \subset AA_1$. So, applying Proposition \ref{propositionAA_1}, we obtain:
$$ \overline {\{\mu_x \colon  x \in AA_2\}} \subset \overline {\{\mu_x \colon  x \in AA_1\}} = \O_f. $$
To obtain the opposite inclusion, we apply \cite[Theorem 1.5]{CE1}):   $\O_f$ is the minimal weak$^*$-compact set of probability measures, that contains $p \omega (x)$ for Lebesgue a.e.~$x$. Since $ \overline {\{\mu_x \colon  x \in AA_2\}}$ is weak$^*$-compact and contains $p\omega_x = \{\mu_x\}$ for Lebesgue almost all $x$ (because $\mbox{Leb}(AA_2) = 1$), we conclude that
    $$   \overline {\{\mu_x \colon  x \in AA_2\}} \supset \O_f.$$

    The inclusion   $\overline {\{\mu_x \colon  x \in AA_2\}} \subset \overline{\mbox{Shr}_f}$ follows trivially from the definition of the set $AA_2$. Now, let us prove the opposite inclusion. We will prove that every shrinking measure is pseudo-physical.  Let $\mu \in \mbox{Shr}_f$. For any  $\e>0$, choose  $q \geq 1$ as in Lemma \ref{lemma-dist}.  By the definition of shrinking measure, $\mu$ is supported on the $f$-orbit  $$K_q  = \bigcup_{j= 0}^{p_q} f^j(\overline I_q) $$ of a periodic shrinking interval $I_q$ of period $p_q$, such that  $\mbox{length}(\overline I_q) <1/q$; hence  $\mbox{length}(f^j(\overline I_q)) <1/q \ \ \forall \ 1 \leq j \leq p_q$, and thus   $\mu(K_q) = 1$.

    Besides, for any point $x \in I_q$,  all the measures of $p\omega_x$ are also supported on $K_q$. In particular, for  $x \in I_q \cap AA_1$ we obtain $\mu_x (K_q) = 1$.
    Finally, applying Lemmas \ref{lemma-dist} and  \ref{lemmaInvMeasOnShrinkingIntervals}, we deduce that $$\mbox{dist}(\mu, \mu_x) < \e \ \mbox{ for any } x \in I_q \cap AA_1. $$ Since $Leb (I_q \cap AA_1) = Leb (I_q) >0$, the basin $A_\e (\mu)$ has positive Lebesgue measure; namely $\mu $ is pseudo-physical.

    We have shown that every shrinking measure is pseudo-physical. Since the set $\O_f$ of pseudo-physical measures is closed, we conclude $$\overline {\mbox{Shr}_f} = \mathcal O_f, $$
finishing the proof of Theorem \ref{theoremSRB-l=ClosureSrh}.
\end{proof}

\begin{theorem} \label{theoremShrIncludedErgodic-SRB-l-includedClosureErgodic}
For any map $f \in \C$, if $\mu \in \mbox{Shr}_f$, then it is ergodic.
\end{theorem}

\begin{corollary}\label{corollaryShrink}
For a typical map $f \in \C$:
$$\O_f  = \overline{\mbox{Shr}_f}\subset \overline{{\mathcal E}_f} \subset  \overline{\mbox{Per}}_f.$$
\end{corollary}

The corollary immediately follows by combining Theorems \ref{theoremShrIncludedErgodic-SRB-l-includedClosureErgodic}, \ref{theoremSRB-l=ClosureSrh}
 with Corollary \ref{corollaryOprocha&als}.
At the end of next section we will prove that  for typical maps these sets are all equal.

\begin{proof}
Fix $f \in \C$. Suppose $\mu \in \mbox{Shr}_f$, and  $\mu_1, \mu_2 \in \mathcal{M}_f$ such that
\begin{equation}
\label{eqn13}
\mu = \lambda \mu_1 + (1 -\lambda)\mu_2, \ \ \ \mbox{with} \ \  0 < \lambda <1,\end{equation}
We shall prove that $\mu_1 = \mu_2 = \mu$; namely  $\mu$ is extremal in the convex compact set of invariant measures; hence ergodic.

Take arbitrary $\e >0$ and fix $q \geq 1$ as in Lemma \ref{lemma-dist}. By the definition of infinitely shrinking measures, there exists a periodic shrinking interval $I_q$, with $\mbox{length}(I_q) < 1/q$, and period $p_q$, whose $f$-orbit $K_q$ supports $\mu$.  The definition of periodic shrinking interval and Lemma \ref{lemmaInvMeasOnShrinkingIntervals} tell us:
$$\mu(f^j(\overline I_q))  = \frac{1}{p_q},  \ \ \ \mbox{length} (f^j(\overline I_q)) < 1/q \ \ \ \ \forall \ 1 \leq j \leq p_q.$$

 From  (\ref{eqn13}), for $i=1,2$ we have $\mu_i(A) \le  \mu(A)$ for any measurable set $A$. Since $\mu(K_q) = 1$, we deduce $\mu_1(K_q) = \mu_2(K_q) = 1 $. Applying  Lemma \ref{lemmaInvMeasOnShrinkingIntervals} we obtain
 $$\mu_1(f^j(\overline I_q))  = \mu_2(f^j(\overline I_q))  = \frac{1}{p_q}\ \ \ \ \forall \ 1 \leq j \leq p_q. $$
So,   Lemma \ref{lemma-dist} implies
$\mbox{dist}(\mu_1, \mu) < \e,$ and $\mbox{dist}(\mu_2, \mu) \ \ < \e$.
As $\e> 0$ is arbitrary, we conclude that $\mu = \mu_1 = \mu_2$; hence $\mu $ is ergodic.
\end{proof}

\section{All ergodic measures are pseudo-physical.}

Slightly abusing language, we  refer to $\mu \in \mbox{Per}_{f}$ as a  \em periodic measure. \em
\begin{definition} \em  \label{definition-qShrinking}
Let $q \geq 1$ and $x_0$ be a periodic point with period $r \geq 1$. We call the periodic (invariant) measure  $$\nu= \frac{1}{r} \sum_{j=0} ^{r-1}\delta_{f^j (x_0)}$$   a \em  $q$-shrinked measure,  \em  if there exists some periodic shrinking interval  $I $, with length smaller than $1/q$, with period $p \geq 1$ such that $\nu$ is supported on   $K \defeq \bigcup_{j= 1}^p f(\overline I)$. From the definition of periodic shrinking interval,    the period $p$ must be smaller or equal than $r$, and must divide $r$.
We denote by $\mbox{Shr}_q\mbox{Per}_{f}$ the set of $q$-shrinked periodic measures.

We say that an invariant measure $\mu$,  \em is $\e$-approached by $q$-shrinked  periodic measures \em if there exists $\nu \in \mbox{Shr}_q\mbox{Per}_{f}$ such that $\mbox{dist}(\mu, \nu) < \e $. We denote by $\mbox{AShr}_{\e, q}\mbox{Per}_{f}$ the set of measures that are  $\e$-approached by $q$-shrinked periodic measures.

\end{definition}

\begin{theorem} \label{theoremBigcapClosShr_q&SRBl-l}
For any map $f \in \C$
$$  \bigcap_{\e >0}\bigcap_{q \geq1}\overline{\mbox{AShr}_{\e, q}\mbox{Per}}_f  \subset \O_f.$$
\em
Notes: The opposite inclusion also holds if $f \in \C$ is typical. We will prove it later, in Theorem \ref{TheoremSRB-l=ClosureErgodic}.

It is standard to check that $\bigcap_{\e >0}\bigcap_{q \geq1}\overline{\mbox{AShr}_{\e, q}\mbox{Per}}_f = \bigcap_{q \geq1}\overline{\mbox{Shr}_{ q}\mbox{Per}}_f$.  Nevertheless, we will not use this latter equality.

\end{theorem}
\begin{proof}
Fix  $\e > 0$. Take $q \geq 1$ as in Lemma \ref{lemma-dist},  such that $1/q < \e$. We will  prove that  for any  $\mu_q  \in \mbox{AShr}_{1/q, q}\mbox{Per}_{f} $, the  basin $A_{2\e}(\mu_q)$ has positive Lebesgue measure.

In fact, for any $\mu_q  \in \mbox{AShr}_{1/q, q}\mbox{Per}_f $, denote by $\nu_q $ a measure in $ {\mbox{Shr}_q\mbox{Per}}_f$  such that $$\mbox{dist}(\mu_q, \nu_q) < 1/q< \e.$$
Consider  the periodic shrinking interval $I$ and the compact set $K$ for $\nu_q$ from   Definition \ref{definition-qShrinking}. From the definition of periodic shrinking interval,  any point $x \in I$ satisfies $p\omega (x) \subset K= \bigcup_{j= 1}^p f^j(\overline I)$. So, any  measure $\mu_x \in p\omega (x)$ is supported on $K$. Also $\nu_q$ is supported on $K$. Thus, applying Lemma \ref{lemmaInvMeasOnShrinkingIntervals}, we deduce that
$$\mu_x( f^j(\overline I)) = \nu_q(f^j(\overline I)) = \frac{1}{p} \ \ \forall \ 1 \leq j \leq p; \ \ \ \mbox{length}(f^j(\overline I)) < \frac{1}{q}.$$
Now, from Lemma \ref{lemma-dist}, we obtain  $\mbox{dist}(\nu_q, \mu_x) < \e$ for all $\mu_x \in p\omega(x)$, for all $x \in I$. Thus, for any $x \in I = I(\nu_q)$
\begin{equation}
\label{eqn14}
\mbox{dist}(p \omega(x), \nu_q)< \e,  \mbox{ hence } \mbox{dist}(p \omega(x), \mu_q)< 2 \e.\end{equation}

Note that when we vary the value of $\e >0$, the value of $q$, and thus also the measures $\nu_q$ and $\mu_q$ and the interval $I$, may change. So, from the above inequality we \em can not deduce \em that each $\mu_q$ is pseudo-physical. Nevertheless, we have proved that for any fixed value of  $\e> 0$ there exists $q \geq 1$ such that Inequality (\ref{eqn14}) holds  for all $\mu_q  \in \mbox{AShr}_{1/q, q}\mbox{Per}_f $.

Now, consider any measure $\mu'_q \in  \overline{\mbox{AShr}_{1/q, q}\mbox{Per}_f}$. Thus, there exists $\mu_q \in \mbox{AShr}_{1/q, q}\mbox{Per}_f$ such that
$$\mbox{dist}(\mu'_q, \mu_q) < \e $$
Combining this with (\ref{eqn14}), we obtain, for  fixed $\e >0$,  a value of $q \geq 1$ such that, for any measure $\mu'_q  \in \overline{\mbox{AShr}_{1/q, q}\mbox{Per}_f}$ there exists an open interval $I$ such that
\begin{equation}
\label{eqn15}
\mbox{dist}(p \omega(x), \mu_q) < 3\e \ \ \ \forall \ x \in I.\end{equation}
So, if $\mu \in \bigcap_{\e >0}  \bigcap_{q \geq 1}\overline{\mbox{AShr}_{\e, q}\mbox{Per}_f}$, then,  for all $\e >0$ there exists an open interval $I$ satisfying assertion (\ref{eqn15}). This proves that $\mbox{Leb}(A_{3\e}(\mu)) >0$ for all $\e >0$; hence $\mu \in \O_f$, as wanted.
\end{proof}
\begin{theorem}
\label{TheoremSRB-l=ClosureErgodic}
For a fixed typical map $f \in \C$,
\begin{equation}
\label{eqn16}
\mbox{Per}_f \subset \bigcap_{q \geq1}\overline{\mbox{AShr}_{\e, q}\mbox{Per}_f} \ \ \forall \ \e>0.\end{equation}
Therefore,
\begin{equation}
\label{eqn17}\bigcap_{\e >0}\bigcap_{q \geq1}\overline{\mbox{AShr}_{\e, q}\mbox{Per}_f} =  \mathcal O_f = \overline{\mbox{AShr}_f} =\overline{ \mathcal E_f }=\overline{ \mbox{Per}_f}. \end{equation}
In particular, any invariant ergodic measure for $f$ is pseudo-physical.
\end{theorem}

\begin{proof}
It is enough to prove   (\ref{eqn16}). In fact,     (\ref{eqn17}) is an immediate consequence of   (\ref{eqn16}), Corollary \ref{corollaryShrink} and Theorem
\ref{theoremBigcapClosShr_q&SRBl-l}.
To prove  (\ref{eqn16}), we first define, for any  $q,r \in \N^+$, the following concept:

\vspace{.2cm}

\begin{definition}
\label{definitionP_q,r} \em

A \em good $q, r$-covering  \em $\mathcal U_{q,r}$ for $f \in \C$, is a finite family of \em open \em intervals   such that

(1)  $\mathcal U_{q,r}$ covers the compact set $\{x \in [0,1]\colon f^r(x) = x\}$.

(2)  $\mbox{length}(U_i)< 1/q$ for any $U_i \in \mathcal U_{q,r}$.

(3) For any $U_i \in \mathcal U_{q,r}$, there exists a periodic shrinking interval $I_i$, with period $p_i \leq r$, with $p_i$ that divides $r$, such that $\overline I_i \subset U_i$.

\vspace{.2cm}

We call a map $f \in \C$ \em a good $q,r$-covered map,  \em  if  there exists a  good $q, r$-covering  $\mathcal U_{q,r}$ for $f$. We denote by ${\mathcal P}_{q,r} \subset \C$ the set of all good $q,r$-covered maps.

\end{definition}

Let us prove that, for fixed $q,r \geq 1$, the set ${\mathcal P}_{q,r}$  is open in $\C$.
Fix $f \in {\mathcal P}_{q,r} $, and denote its good $q,r$-covering by  $\mathcal U_{q,r} = \{U_1, U_2, \ldots, U_h\}$. The compact set $K= [0,1] \setminus \bigcup_{i= 1}^h U_i$ does not intersect the compact set $\{f^r(x) =x\}$. Let us prove that for all $g \in \C$ close enough to $f$, the same compact set $K$ (defined for the \em same \em covering ${\mathcal U}_{q,r}$) does not intersect   $\{g^r(x) = x\}$.  In fact,    the real function   $\phi_f(\cdot) : = \mbox {dist} (f^r(\cdot), \cdot) $  depends continuously on $f$. Since $\min _{x \in K} \phi_f(x) >0$, we deduce $\min _{x \in K} \phi_g(x) >0$ for all $g \in \C$ close enough to $f$.
In other words, ${\mathcal U}_{q,r}$ also covers the fixed points of $g^r$.  Thus, the good $q,r$-covering of $f$, is also a covering satisfying conditions (1) and (2) of Definition \ref{definitionP_q,r},  for any $g \in \C$ close enough to  $f$.
Now, let  us prove that Condition (3) for $g$  is satisfied   by the same   covering ${\mathcal U}_{q,r}$, provided that $g $ is close  enough to
 $f$.
Consider a $f$-shrinking periodic interval $I_i \subset U_i \in  {\mathcal U}_{q,r}$, of period $p_i$.
Now  $I_i$ is a periodic shrinking interval with the same period $p_i$  for all $g$ sufficient close  to $f$. Since the finite family $\{I_i\}_{1 \leq i \leq h}$ of shrinking periodic intervals to be preserved is finite, we conclude that
(3) is also satisfied for any $g$ sufficiently close to $f$
and thus ${\mathcal P}_{q,r}$ is open in $\C$.

In Lemma \ref{lemmaP_qrIsDense}, we will  prove that  ${\mathcal P}_{q,r}$ is  dense in $\C$. Let us conclude the proof of Theorem \ref{TheoremSRB-l=ClosureErgodic} assuming Lemma \ref{lemmaP_qrIsDense}.
Observe that, for fixed $q,r \geq 1$, any $f \in {\mathcal P}_{q,r}$ has the following property:
any   point $x_0$ fixed by $f^r$ (in particular any periodic point $x_0$ of period $r$) is $(1/q)$-near all the points of a periodic shrinking interval $I_0$ with length smaller than $1/q$, and with period $p_0  \leq r$, $p_0 $ dividing $r$.

Besides, any periodic shrinked interval of period $p_0$ has at least one periodic point $y_0$, fixed by $f^{p_0}$. We deduce that $I_0 $, whose length is smaller than $1/q$, contains  a periodic point $y_0$. Now, using the  definition of the set of measures $\mbox{Shr}_q \mbox{Per}_f$, we summerize this assertion as follows:
 \begin{equation}
\label{eqn18} \forall \ x_0 \mbox{ with period } r, \ \  \ \exists \
\nu_0 \defeq  \frac{1}{r} \sum_{j= 0}^{r-1} \delta_{f^j(y_0)} \in \mbox{Shr}_q \mbox{Per}_f, \end{equation} $$
\mbox{ with } \ \  |y_0- x_0| < 1/q.$$

For   $r \geq 1$ fixed, consider $f \in \bigcap_{q \geq 1} \mathcal P_{q,r}$. For any $\e>0$, apply Lemma \ref{lemma-dist} to find $q' \geq 1$ such that:
$\mbox{ if } |x-y| < 1/q' < \e$  then  $\mbox{dist}(\delta_x, \delta_y) < \e$.
Choose $0<\delta < 1/q' $ such that
$\mbox{if }|x-y| < \delta$  then $|f^j(x)- f^j(y)| < 1/q' \ \ \  \forall \ 0 \leq j \leq r$.
Let us consider all  $q \in \N$ such that $q > 1/\delta > q'$.
Remember that $f  \in   \mathcal P_{q,r}$ for all $q \geq 1$. We apply Assertion (\ref{eqn18}), since $|y_0 - x_0| < 1/q$,  we  deduce
  $$|f^j(y_0) - f^j(x_0)| < 1/q' , \ \ \ \ \ \mbox{dist}(\delta_{f^j(y_0)}, \delta _{f^j(x_0)}) < \e \ \ \forall \ 0 \leq j \leq r.$$
Since  balls in the space of probability measures with the weak$^*$-topology  are convex, we conclude
  $$\mbox{dist} \left ( \frac{1}{r} \sum_{j=0}^{r-1}\delta_{ f^j(y_0)}, \ \   \frac{1}{r} \sum_{j=0}^{r-1} \delta_{ f^j(x_0)}   \right ) < \e.$$
We have shown that for any given periodic orbit $\{f^j(x_0)\}_{ 0 \leq j \leq r-1}$ of period $r$, the distance between the periodic measure  supported on it, and some measure $\nu_q \in \mbox{Shr}_q \mbox{Per}_f$, for all $q$ small enough, is smaller than $\e$.
  In other words, any periodic measure supported on a periodic orbit of period $r$, belongs to $\bigcap_{q \geq 1} \overline {\mbox{AShr}_{\e, q}\mbox{Per}_f}$.

  Finally, if $f \in {\mathcal P}\defeq \bigcap_{r \geq 1} \bigcap_{q \geq 1}\mathcal P_{q,r}$, then all its periodic measures (supported on periodic orbits of any period $r$) will belong to $\overline{\mbox{AShr}_{\e, q}\mbox{Per}_f} $. In brief, if $f \in   \mathcal P$, then
  $$\mbox{Per}_f \subset \overline{\mbox{AShr}_{\e, q}\mbox{Per}_f}  \ \ \ \forall \ \e > 0. $$

   As ${\mathcal P}_{q,r}$ is open and dense in $\C$, the maps $f \in \mathcal P$ are typical. This ends the proof of Theorem \ref{TheoremSRB-l=ClosureErgodic} (provided that Lemma \ref{lemmaP_qrIsDense} is proven).
\end{proof}
\begin{lemma}
\label{lemmaP_qrIsDense}

For each  $q,r \geq 1$, the set ${\mathcal P}_{q,r}$
is dense in $\C$.
\end{lemma}

\begin{proof}   Fix $f \in \C$ and $\e>0$.
Take $0<\delta \defeq \frac{1}{q'}< \min\Big \{\frac{1}{q}, \ \frac{\e}{2}  \Big \}$  such that $|x-y| < \delta \ \Rightarrow \ |f(x) - f(y)|< \frac{\e}{2}$.
To $\e$-perturb $f$  into new map $g \in {\mathcal P}_{q,r}$, we will proceed
in a finite number of steps. For each $i \in \{0,\dots, 2q'-2\}$ consider the open interval $J_i \defeq (\frac{i}{2q'}, \frac{i+2}{2q'})$, this collection of intervals forms an open cover $\mathcal J$ of $[0,1]$.

\noindent{\bf First step: }  Consider the finite covering
$\J_1  = \{J_{1,1}, \ldots, J_{1,h_1}\} \subset \mathcal J$ of  $J \in \mathcal J_1$ which intersect the compact set $K_1\defeq \{f^r(x)  = x\}$.
Construct   $\mathcal J_1$ to be a minimal covering of $K_1$ in the following sense: in each interval   $J_{1,i} \in \mathcal J_1$, there exists at least one point $x_{1,i}$ such that $$ x_{1,i} \in J_{1,i}, \ \ f^r(x_{1,i} ) = x_{1,i},  \ \ x_{1,i} \not \in  \overline {J_{1,j}} \ \ \forall \ j \neq i.$$
This is possible since if all the $f^r$-fixed points  that belong to the open interval $J_{1,i}$, also belonged   to $ \overline {J_{1,j}}$ for some $j \neq i$, we can slightly enlarge  $J_{1,j}$ to be sure that all of them belong to the interior of  $\overline {J_{1,j}}$, and would suppress $J_{1,i}$ from the covering $\J_1$. We can do this in such a
way that the length of all the intervals are at most $\delta_1 < \min \{\frac{1}{q},\frac{\e}{2}, 1.1 \cdot  \delta\}$.
Note that $\mathcal J_1$ is no longer a sub-collection of $\mathcal J$, but it is naturally in a bijection with a sub-collection. Throughout the proof we will use this
bijection, and make further modifications to the intervals in such a way that the new collections are always in bijection with a sub-collection of $\mathcal J$.

Denote    $$F_1 : = \{x_{1,i}: 1 \leq i \leq h_1\}, \ \ \ \ \ (F_1)^r \defeq \bigcup_{j= 0}^{r-1} f^j(F_1).$$
Hence $F_1$ is consisits of exactly  $h_1$ different points that are fixed by $f^r$, and $(F_1)^r$  consists of at most $r h_1$ different points, also fixed by $f^r$.

In each open interval $J_{1,i} \in \mathcal J_1$ construct  a small open interval $I_{1,i}$ such that $$x_{1,i} \in I_{1,i}, \ \ \ \  \overline {I_{1,i}} \subset V_{1,i} \defeq J_{1,i}\setminus \bigcup_{j \neq i} \overline {J_{1,j}}, \ \ \ \   (F_1)^r    \bigcap \  \overline{I_{1,i}} = \{x_{1,i}\}.$$
For each $1 \leq i \leq h_1$, define $$g:  \overline {V_{1,i}} \mapsto \Big(f(x_{1,i}) - \frac{\e}{2}, \ \ \  f(x_{1,i}) - \frac{\e}{2} \Big )$$ to be continuous, piecewise affine,  with a minimum number of affine pieces, and  such that

(a) $g(x) = f(x)$ if $x \not \in V_{1,i}$ or if $x \in V_{1,i} \bigcap  (F_1) ^r$.

\hspace{.5cm}  In particular $g(x_{1,i})= f(x_{1,i})$.

(b) $g(x) = f(x_{1,i})$ if $x \in \overline {I_{1,i}}.$

\vspace{.2cm}

 Let us check that $|g(x) - f(x)| < \e$ for all $x$.
First of all by construction, $g(x) = f(x)$ for all $x \not \in \bigcup_{i= 1}^{h_1} V_{1,i}$.
Now suppose   $x \in V_{1,i}$ for some $1 \leq i \leq h_1$.
Since $g|_{\overline{V_{1,i}}}$ is piecewise affine with a minimum number of pieces satisfying conditions (a) and (b),
 there exists two points  $y , z \in \overline{V_{1,i}}$  such that $g(y) = f(y) \leq g(x) \leq g(z) = f(z)$. Thus,
 $$ g(x) - f(x) \leq f(z) - f(x) \leq |f(z) - f(x)| <\e , \  \ \mbox{ and } $$ $$g(x) - f(x) \geq f(y) -  f(x) \geq -|f(y) - f(x)|  > - \e. $$

 By construction, each interval $I_{1,i}$ is periodic shrinking for $g$ with period that divides $r$. In fact, $g(\overline{I_{1,i}}) = f(x_{1,i})$, the orbits by $g$ and $f$ of ${x_{1,i}}$  coincide, and  so $g^r(x_{1,i}) = f^r(x_{1,i}) = x_{1,i} \in I_{1,i}$. We deduce that $I_{1,i}$ is periodic shrinking for $g$, with some period equal to the period  of the point  $x_{1,i}$, which divides $r$.

  Up to now, we have constructed an $\e$-perturbation $g$ of $f$ and a finite family $\J_1$ of open intervals that satisfy Conditions (2) and (3) of Definition \ref{definitionP_q,r} for $g$, and Condition (1) for $f$.  Either $\J_1$ also satisfies condition (1) for $g$, or not. In the first case, define  ${\mathcal U}_{q,r}\defeq\J_1$. This covering satisfies Definition  \ref{definitionP_q,r} for $g$. Hence, it  is a good $q,r$-covering for $g$. So $g \in {\mathcal P}_{q,r}$ and the proof of Lemma \ref{lemmaP_qrIsDense} is finished. In the second case,
  we will continue in Step 2, to modify $f$ in $[0,1]\setminus   { \bigcup_{J  \in \J_1} J }$.

 \noindent{\bf Second step: }  Construct a  finite covering $\mathcal J_2$ of the compact set $K_2\defeq \{g^r(x)  = x\}$ by adding  to $\J_1$ all the intervals of $\J$ not  ``bijectively'' in $\J_1$ which intersect the compact set $K_2\defeq \{g^r(x)  = x\}$.

   Choose  the  new open intervals $J_{2,i}$, with $\ 1 \leq i \leq h_2$, such  that $\mathcal J_2$
is a minimal covering of $K_2$ in the following sense: in each interval   $J_{2,i} \in \mathcal J_2 \setminus \J_1$, there exists at least one point $x_{2,i}$ such that
$$ x_{2,i} \in J_{2,i}, \ \ g^r(x_{2,i} ) = x_{2,i},  \ \ x_{2,i} \not \in  \overline {J_{2,j}} \ \ \forall \ j \neq i, \ \ x_{2,i} \not \in   { \bigcup_{J  \in \J_1} J }.$$
Besides, take $$x_{2,i} \not \in \overline{{ \bigcup_{J  \in \J_1} J }} \ \ \ \forall \ 1 \leq i \leq h_2.$$
In fact, if for some $i$ all the points that are fixed by $g$ and belong to $J_{2,i} \setminus { \bigcup_{J  \in \J_1} J }$, also belonged to $\partial { (\cup_{J  \in \J_1} J )}$, we would slightly  enlarge the corresponding interval $J_{1,j} \in \mathcal J_1$ so its boundary has not   fixed points by $g^r$. We would  not be changing $g$ anywhere yet; but just slightly enlarging an open interval $J_{1, j}$ of the old covering $\mathcal J_1$ that was constructed in step 1. After that, we would remove the interval $J_{2,i}$ from the new covering $\mathcal J_2$.

  Denote    $$F_2 : = \{x_{2,i}: 1 \leq i \leq h_2\}, \ \ \ \ \ (F_2)^r \defeq \bigcup_{j= 0}^{r-1} f^j(F_2).$$
In each  $J_{2,i} \in \mathcal J_2 \setminus \mathcal J_1$ choose an open interval $I_{2, i}$ such that $$x_{2,i} \in I_{2,i}, \ \ \ \  \overline {I_{2,i}} \subset V_{2,i} \defeq J_{2,i}\setminus \Big(\big(\overline{\cup_{J \in \J_1} J}\big)  \cup \ \big(\cup_{j \neq i} \overline {J_{2,j}}\big )\Big),  \mbox{ and}$$ $$  (F_2)^r    \bigcap \  \overline{I_{2,i}} = \{x_{2,i}\}.$$
Observe that the function $g$ that was constructed on the first step, coincides with $f $ in $  \overline{V_{2,i}}$ for all $ 1 \leq i \leq h_2.$ Now, we are going to change $f$, only inside each open interval $V_{2,i}$, as follows:

For each $1 \leq i \leq h_1$, define a continuous piecewise affine map  $$g:  \overline {V_{2,i}} \mapsto \Big(f(x_{2,i}) - \frac{\e}{2}, \ \ \  f(x_{2,i}) - \frac{\e}{2} \Big )$$   with a minimum number of affine pieces, and  such that

(a) $g(x) = f(x)$ if $x   \in \partial V_{2,i}$ or if $x \in V_{2,i} \bigcap  (F_2) ^r$.

\hspace{.5cm} In particular $g(x_{2,i})= f(x_{2,i})$.

(b) $g(x) = f(x_{2,i})$ if $x \in \overline {I_{2,i}}.$

\vspace{.2cm}

 Similarly as we argued in the first step,  we have $|g(x) - f(x)| < \e$ for all $x \in V_{2,i}$ for all $1 \leq i \leq h_2$ (while we already had $|g(x) - f(x)| < \e$ for all $x \not \in \bigcup_{i= 1}^{h_2} V_{2,i}$).

 By construction, each interval $I_{2,i}$ is periodic shrinking for $g$ with period that divides $r$. Thus we have constructed an $\e$-perturbation $g$ of $f$ and a finite family $\J_2$ of open intervals that satisfies Conditions (2) and (3) of Definition \ref{definitionP_q,r} for $g$.  Either $\J_2$ also satisfies Condition (1) for $g$, or not. In the first case, define  ${\mathcal U}_{q,r}\defeq\J_2$  and the proof of Lemma \ref{lemmaP_qrIsDense} is finished. In the second case,
  we should continue with Step 3, and construct the collection $\J_3$ in a similar way.

  Note that the procedure must finish in a finite number of steps, because in the passage from the collection $\J_i$ to the collection $\J_{i+1}$ we add at least one interval of  the finite collection $\J$.
\end{proof}

\section{pseudo-physical measures with finite large entropy}

\begin{definition}
\label{DefinitionHorseshoe} \em

Fix  $f \in \C$ and $m \geq 2$.
An \em $m$-horseshoe  for $f$ \em is a family of $m$ pairwise disjoint closed intervals $\overline I_{i}, \  1 \leq i \leq m$, with nonempty interiors $I_i$, such that
\begin{equation} \label{eqn20}\mbox{int}(f( I_i)) \supset \overline I_1 \cup \overline I_2 \cup \ldots \cup \overline I_m \ \ \ \forall \ 1 \leq i \leq m.\end{equation}
\end{definition}
Condition (\ref{eqn20}) is persistent under small perturbations of the map $f$, i.e., if $\{\overline I_{i}\}_{1 \leq i \leq m}$ is an $m$-horseshoe  for $f $, then, the \em same \em family of intervals is also an $m$-horseshoe for all $g \in \C$ close  enough to $f$.

\begin{definition} \label{DefinitionAtomsOfHorseshoe} \em We   define the  {\em atoms of a horseshoe} in a non-standard way.

\noindent{\bf Atoms of generation 1. } Let $\{\overline I_{i}\}_{1 \leq i \leq m}$ be an $m$-horseshoe for $f$. We call each interval $\overline I_i$ \em an atom of generation   \em 1 of the horseshoe.  We denote by ${\mathcal A}_1$ the family of  atoms of generation 1. We have $\#{\mathcal A}_1 = m$.

\noindent{\bf Atoms of generation 2. }
By the definition of horseshoe, we have 
at least one closed interval $\overline I_{i, j} \subset  I_i$, with nonempty interior $I_{i,j}$, such that $\mbox{int}(f(  I_{i,j})) \supset  \overline I_j$.
For each  $i,j$ we \emph{choose} one and only one interval $\overline I_{i,j}$ satisfying this condition. We call such an interval, \em an atom of generation \em 2  of the horseshoe. We denote by ${\mathcal A}_2$ the family of  atoms of generation 2. Note that $\#{\mathcal A}_2 = m^2$.

%
%
 \noindent{\bf Atoms of generation n. } The family ${\mathcal A}_n$ of atoms of generation $n \geq 2$
  is composed by $m^n$ pairwise disjoint compact intervals with nonempty interiors, such that  each  $ \overline I_{i_1, i_2, \ldots, i_n} \in {\mathcal A}_n$ is identified by a different $k$-uple $(i_1, i_2, \ldots, i_n) \in \{1, 2, \ldots, m\}^n$, and satisfies the following properties:
$$\overline I_{i_1, i_2, \ldots, i_n}  \subset  I_{i_1, i_2, \ldots, i_{n-1}} \cap f^{-1} (I_{i_2, \ldots, i_{n-1}}) \subset  I_{i_1} \cap f^{-1}(I_{i_2}) \cap \ldots \cap f^{-(n-2)}(I_{n-1}) ;  $$ $$
 \mbox{int}(f(\overline I_{i_1, i_2, \ldots, i_n})) \supset  \overline I_{i_2, \ldots, i_n} \in {\mathcal A}_{n-1}.
 $$

For each generation $n \geq 2$,  once the family ${\mathcal A}_{n-1}$ is defined, there may exist many possible choices of the family ${\mathcal A}_{n}$ of atoms of the next generation, since there may exist many adequate connected components of  the set  $ I_{i_1, i_2, \ldots, i_{n-1}} \cap  f^{-1} (I_{i_2, \ldots, i_{n-1}}).$
\emph{Throughout the article when we speak of a horseshoe, we fix such a choice.}
\end{definition}
\begin{definition} \em
Fix  $m \ge 2$ and $f \in \C$ exhibiting an $m$-horseshoe.
The  \emph{$\Lambda$-set} of the horseshoe is defined by
$$\Lambda_n \defeq \bigcup_{A \in {\mathcal A}_n} A, \ \ \ \ \
 \Lambda = \bigcap_{n \geq 1} \Lambda_n.$$

 Observe that  $\Lambda_n$ is compact and $\Lambda_{n+1} \subset \Lambda_n$ for all $n \geq 1$. Thus, $\Lambda$ is a nonempty compact set. It is a perfect set. Nevertheless $\Lambda$ is not necessarily totally disconnected since it may contain intervals, because the diameters of the atoms of generation $n$ do not necessarily go to zero as $n \rightarrow + \infty$. So, $\Lambda$ is not necessarily a Cantor set.

\end{definition}

\begin{definition} \em \label{definitionHyperbolicHorseshoe}

Fix $m \ge 2$ and $f \in \C$ with an $m$-horseshoe.
We call the $m$-horseshoe
 \em $C^0$-hyperbolic \em if there exists  $\lambda \in (0,1)$ such that
$$\max_{A \in {\mathcal A}_n} \mbox{length}(A) < \lambda^n \ \ \ \ \ \forall \ n \geq 1.$$

%

Note that if the horseshoe is $C^0$-hyperbolic, and more generally if the maximum diameter of the atoms of generation $n$ goes to zero as $n \rightarrow + \infty$, then the  $\Lambda$-set of the horseshoe is a Cantor set.
\end{definition}

The following is a well known result on the existence of Benoulli measures supported on $\Lambda$-sets of hyperbolic horseshoes:
\begin{proposition} \label{propositionBernoulliMeasOnHorseshoe}
Assume that $f \in \C$ exhibits a $C^0$-hyperbolic $m$-horseshoe. Then, there exists an $f$-invariant ergodic measure $\mu$ supported on the  $\Lambda$-set of the $m$-horseshoe such that $$\mu(A) = \frac{1}{m^n} \  \ \ \forall \ A \in {\mathcal A}_n, \ \ \ \ \forall \ n \geq 1, \ \ \mbox{and}$$
$$h_{\mu}(f) = \log m.$$
\em
\end{proposition}
 The measures $\mu$ of Proposition \ref{propositionBernoulliMeasOnHorseshoe} are called \em   Bernoulli measures. \em

\begin{theorem} \label{theorempseudo-physicalMeasEntropyLog_m}

A typical map $f \in \C$ has the following properties:

\vspace{.2cm}

\noindent {\rm (i) } For any  $x_0 \in [0,1]$ fixed by $f$,   for any $\e> 0$, and for any $m \geq 2$, there exists a $C^0$-hyperbolic $m$-horseshoe
supported on a Cantor set contained in the $\e$-neighborhood of $x_0$.

\vspace{.2cm}

\noindent {\rm (ii) }For any  $m \geq 2$, there exists infinitely many distinct Bernoulli measures with entropy equal to $\log m$. Every Bernoulli measure is
 ergodic and hence pseudo-physical for $f$.
\end{theorem}
\begin{proof}
A typical map has infinitely many periodic points.  Thus
 (ii) is a direct consequence of  (i), Proposition \ref{propositionBernoulliMeasOnHorseshoe} and Theorem \ref{TheoremSRB-l=ClosureErgodic}. So, it is enough to prove (i).  To do so, we  need the following definition:

\begin{definition} \em
\label{definitionB_{q,n,m}} Fix  $q,   m \geq 2$.
We say that  $f \in \C$ belongs to the class ${\mathcal B}_{q,m} \subset \C$ if
there exists a finite covering with open intervals $J_1, J_2, \ldots, J_h $ of the set of fixed points of $f$, such that:

a) $\mbox{length}(J_i) \leq 1/q$ for all $1 \leq i \leq h$; and

b) each interval $J_i$ contains a $C^0$-hyperbolic $m$-horseshoe.
\end{definition}
To prove Theorem \ref{theorempseudo-physicalMeasEntropyLog_m}(i), it is enough to prove that, for each fixed   $q, m \ge 2$, the class ${\mathcal B}_{q, m}$ is a dense  $G_{\delta}$-set in $\C$.

First, let us see that ${\mathcal B}_{q,m}$ is a $G_{\delta}$ set in $\C$. The empty set is a $G_\delta$ set by definition, thus suppose
 ${\mathcal B}_{q,m} \neq \emptyset$, and take $f \in {\mathcal B}_{q,m}$. As already argued in the proof of Theorem \ref{TheoremSRB-l=ClosureErgodic}, the same covering $\{J_1, J_2, \ldots, J_h\}$ of the set of fixed points by $f$,  also covers the set of fixed points by $g$, for any map $g \in \C$ that is close  enough to $f$, say in the open neighborhood $U_0(f) \subset \C$.
 Since $f \in  {\mathcal B}_{q, m}$, for each $1\leq i \leq h$ here exists an $m$-horseshoe $H_i \subset J_i$ .
 Consider  the $m$ compact and pairwise disjoint intervals   $\overline I_1^{i}, \ldots, \overline I_m^{i}$ of the $m$-horseshoe $H_i$. After  Condition (\ref{eqn20}) in the definition of horseshoe, the same family of compact and pairwise disjoint intervals, is also a horseshoe of all $g$ close  enough to $f$, say in the open neighborhood $U^i_1(f) \subset U_0(f)$ of $f$ in $\C$.

Now, let us prove that if $H_i$ is  $C^0$-hyperbolic  for $f$, then it is also  $C^0$-hyperbolic up to level $n$, with the same constant $\lambda \in (0,1)$,  for all $g \in \C$ close  enough to $f$, say in the open neighborhood $U^i_n(f) \subset U^i_{n-1}(f) \subset \ldots \subset U^i_1(f) \subset U_0(f)$.
By hypothesis, for all $1 \leq k \leq n$, any atom $\overline I^i_{i_1, i_2, \ldots, i_k}  \in {\mathcal A}^i_k(f)$ of the horseshoe $H_i$ for $f$, has length strictly smaller than $\lambda_i^k$. From the definition of atoms:
\begin{equation}\label{short}
\overline I^i_{i_1, \ldots, i_k} \subset I^i_{i_1,  \ldots, i_{k-1}} \cap f^{-1} (I^i_{i_2, \ldots, i_{k-1}}),  \ \mbox{int} (f( I^i_{i_1, i_2, \ldots, i_k}) )\supsetneq \overline I^i_{i_2, \ldots, i_k}.
\end{equation}

This finite number of open conditions persists under small perturbations of the map $f$, with the \em same family of atoms \em up to generation $n$, i.e.,
Equation \eqref{short} holds for  for all $g \in \C$ close  enough to $f$, say $g \in U_n^i(f) \subset \C$, for the  \em same \em families ${\mathcal A}^i_k(f)$ of intervals.

 These properties imply that we can choose ${\mathcal A}^i_k(g) = {\mathcal A}^i_k (f)$, for all $1 \leq k \leq n$, if $g \in U_n^i(f)$.  Furthermore, we can choose $U_n^i(f)$ so small such that  the lengths of the atoms of generation $k$ of the horseshoe $H_i$ for any $g \in U_n^i(f)$ are also strictly smaller than $(\lambda_i)^k$, for all $1 \leq k \leq n$, with fixed $n \geq 1$. In other words,  the \em same \em horseshoe $H_i$ for $f$, is $C^0$-hyperbolic up to level $n$, for any $g \in U_n^i(f)$.  Finally define
  $$U_n(f)\defeq \bigcap_{i= 1}^hU_n^i(f) \mbox{  and  } G\defeq \bigcap_{n \geq 1, }^{+\infty} \bigcup_{f \in {\mathcal B}_{q,m}}U_{n}(f).$$
 Trivially ${\mathcal B}_{q,m} \subset G$, and by construction $G \subset {\mathcal B}_{q,m}$, i.e.,  ${\mathcal B}_{q,m}$ is a $G_{\delta}$-set.

Now, let us prove that ${\mathcal B}_{q,m}$ is    dense    in $\C$.  Fix  $f \in \C$ and $\e >0$. Let us construct $g \in {\mathcal B}_{q, m}$ such that $\mbox{dist}(g,f)< \e$.
Fix   $\delta \in (0,\max\{1/q, \e/3\})$ such that, if $d(x,y) < \delta$, then  $d(f(x), f(y)) <  \nolinebreak \e/3$.  Consider a finite covering $\J = \{J_1, J_2, \ldots J_h\}$ of the set of fixed points of $f$, with open intervals $J_i$ of length smaller than $\delta$. Choose this covering to be minimal in the following sense:   each interval $J_i \in \J$ has at least one fixed point by $f$ that does not belong to $\bigcup_{j \neq i} J_j$, and besides $J_i \setminus \bigcup_{j \neq i} J_j$ is a single interval which we will denote $(y_i,z_i)$.

We will change $f$, to construct a new map $g \in \C$, such that $g(x) = f(x)$ for all $x \not \in \bigcup_{i=1}^h J_i$. So, the same family $\J$ will be also be a covering of the fixed points of $g$.

In each open interval $J_i= (a_i, b_i)$ choose an open interval $V_i$ such that  $   V_i \subset J_i \setminus (\cup_{j \neq i} J_j)$, and take $m+1$ points $x_{i,0} < x_{i,1} < \ldots <x_{i,m}  $ in $V_i$. Construct the continuous piecewise affine map $g|_{\overline J_i}$, with a minimum number of affine pieces, such that:

$g(x) = f(x)$ if $x \in \partial \big(J_i \setminus (\cup_{j \neq i}  J_j)\big);$

$g(x_{i,l}) = a_i$  if $l$ is even;

$g(x_{i,l}) = b_i$  if $l$ is odd.

The set $[x_{i,0}, x_{i,m}] \bigcap g^{-1} \big( [x_{i,0}, x_{i,m}]  \big)$ has $m$ connected components, they  are compact intervals. We can  slightly enlarge these  intervals to create an  $m$-horseshoe for $g$ contained in $V_i \subset J_i$.
So,  $g \in {\mathcal B}_{q,m}$.

Next we will show that  $\mbox{dist}(g,f)< \e$.
For any   point $x  \in J_i \setminus (\cup_{j \neq i} J_j) \defeq (y_i, z_i)$, we have either (i) $y_i < x \leq x_{i,0}$, or (ii) there exists $0 \leq l \leq m-1$   such that $x_{i,l} < x \leq x_{i, \  l+1}$,
or (iii) $x_{i,m} < x \leq z_i$.
 In case (i) let $x = \alpha x_{i,0} + (1-\alpha) y_i$ for some $\alpha \in [0,1]$.  By linearity
 $g(x) = \alpha g(x_{i,0}) + (1-\alpha) g(y_i)$.  Remembering that $g(y_i) = f(y_i)$ yields
 \begin{equation*}
\begin{split}
|g(x) - f(x)| & \leq |g(x) - g(y_i)| + |g(y_i) - f(x)|\\
& = \alpha |  g(x_{i,0})  -  g(y_i) | + |f(y_i) - f(x)|.
\end{split}
\end{equation*}

Since both $y_i$ and $x$ belong to the interval $[a_i,b_i]$ which has length at most $\delta$, the second term is less than $\e/3$, so we need to give an upper bound on the first term.  We use the fixed point $f(x_i) = x_i$ and once again the fact that $f(y_i) = g(y_i)$ to obtain
\begin{equation*}
\begin{split}
 |  g(x_{i,0})  -  g(y_i) | & \le  |g(x_{i,0}) - x_i| + |f(x_i) - f(y_i)|\\
 & =  |a_i - x_i| + |f(x_i) - f(y_i)|  < \delta + \frac{\e}3.
\end{split}
\end{equation*}
Combining these estimates yields $|g(x) - f(x)| < \e$ as needed. 
Case (iii) is similar to case (i), using $z_i$ instead of $y_i$ and $x_{i,m}$ instead of $x_{i,1}$. Analogously, for case (ii) use  $x_{i,l}$ instead of $y_i$ and $x_{i,l+1}$ instead of $x_{i,1}$.

Recall that  from the choice of the covering $\J$, for any $1 \leq i \leq h$ there exists a point $x_i \in   J_i \setminus (\cup_{j \neq i} J_j)$ such that $f(x_i)= x_i$. Thus, in
Case (ii) , we write:
\begin{equation*}
\begin{split}
|g(x) - f(x)| & \leq | g(x) - f(x_i)| + |f(x_i) - f(x)| \\
&=  | g(x) - x_i| + |f(x_i) - f(x)| \\
& \leq   |b_i -a_i|+ \frac{\e}{3} \leq \delta + \frac{\e}{3} < \e.
\end{split}
\end{equation*}

We conclude that, for fixed natural numbers $q,n,m \geq 2$, for all $f \in \C$ and for all $\e >0$, there exists $g \in {\mathcal B}_{q,m}$ such that $\mbox{dist}(g,f) < \e$. In other words, the family ${\mathcal B}_{q,m}$ of maps is dense in $\C$.
\end{proof}

\begin{corollary} \label{CorollaryNoSemiContinuity}
For a typical map $f \in \C$ the entropy function $\mu \in {\mathcal M}_f \mapsto h_{\mu}(f) \in [0, +\infty]$ is neither upper nor lower semi-continuous. Moreover,  if restricted to the set of pseudo-physical measures, the entropy function is neither upper nor lower semi-continuous.
\end{corollary}
\begin{proof} We apply Theorem \ref{theorempseudo-physicalMeasEntropyLog_m}, fix  $m \geq 2$, and an atomic invariant measure $\delta_{x_0}$, where $x_0$ is a fixed point of $f$. There exists a sequence $\mu_n$ of Bernoulli invariant measures supported on  $\Lambda$-sets of $m$-horseshoes, that are arbitrarily near the fixed point $x_0$. From Lemma \ref{lemma-dist}, $\lim \mu_n = \delta_{x_0}$ in the weak$^*$ topology.      We deduce that $0 = h_{\delta_{x_0}} (f) <\lim_n h_{\mu_n}(f) = \log m $, with $\mu_n \rightarrow \delta_{x_0}$, proving  that the entropy function is not upper semi-continuous.
 For the upper semi-continuity statement restricted to pseudo-physical mesures it suffices to remember that
Bernoulli measures $\mu_n$ are  ergodic, hence pseudo-physical.

On the other hand, each Bernoulli measure $\mu$ supported on the $\Lambda$-set of an $m$-horseshoe  is ergodic. Thus by Theorem \ref{TheoremSRB-l=ClosureErgodic} $\mu $ is pseudo-physical, and there exists a sequence of measures $\nu_n  \in \mbox{Per}_f $ convergent to  $ \mu$ in the weak$^*$-topology. We conclude that $\log m = h_{\mu}(f) > \lim _n h_{\nu_n}(f) = 0$, with $\nu_n \rightarrow \mu$, proving that the entropy function is not lower semi-continuous. For the lower semi-continuity statement restrict to pseudo-physical mesures it suffices to remember that each measure in $\mbox{Per}_f$ is ergodic, and thus pseudo-physical.
\end{proof}

\section{pseudo-physical measures with infinite entropy}

 In the sequel we will use the following notation:

\noindent  $\bullet$  $[t]_n$ or $[s]_n$ will denote   triangular ``matrices''  with $n $ rows as follows: $$[t]_n  = \left[
                                                \begin{array}{cccccc}
                                                  t_{1,1} & t_{1,2} &t_{1,3} &\ldots & t_{1, n-1} & t_{1,n} \\
                                                  t_{2,1} & t_{2,2}  &\ldots & t_{2, n-2}& t_{2,  n-1} &   \\
                                                  t_{3,1} &\ldots &t_{3,n-3}   & t_{3, n-2} & &   \\
                                                    \ldots &\ldots &\ldots   & & &   \\
                                                    \ldots & \ldots& & & &  \\
                                                  t_{n,1}&   &   &  & &  \\
                                                \end{array}
                                              \right], $$ where  $t_{i,j} \in \{0,1\}$ for all $(i,j)$ such that $1 \leq i \leq n, \ 1 \leq j \leq n+1-i$.

\noindent   In particular $[t]_1$ denotes a single number $t_{1,1} \in \{0,1\}$.

\noindent  $\bullet$  $T_n$ will denote the set of all the triangular matrices with $n$ rows described above. Note that $\# T_n = 2 ^{n(n+1)/2}$.

\noindent  $\bullet$  $\pi_n$ and $\sigma_{n}: T_{n+1} \mapsto T_{n}$  will denote the following transformations:
$$\pi_n  [t]_{n+1}  = [s]_{n}, \mbox{ where } s_{i,j} = t_{i,j} \ \ \forall \ 1 \leq i \leq n, \ \ 1 \leq j \leq n+1-i; $$
i.e.,  erase the last diagonal at the right of the triangle matrix $[t]_{n+1}$.
$$\sigma_n  [t]_{n+1}  = [s]_{n}, \mbox{ where } s_{i,j} =  t_{i+1, j} \ \ \forall \ 1 \leq i \leq n, \ \ 1 \leq j \leq n+1-i.$$
i.e., we erase the first column at left of $[t]_{n+1}$.
We call $\pi_n$ the \em projection, \em and $\sigma_n$ the \em shift \em (to the left).

\noindent  $\bullet$   $\overline I_{[t]_n}$ denotes a collection  of $n(n+1)/2$ pairwise disjoint
compact intervals with nonempty interior indexed by the collection of matrices $[t]_n \in T_n$,
and  $ I_{[t]_n}$ denotes the collection of  interiors of these intervals.

\begin{definition} \em
\label{definitionCascade}
An {\em atom doubling cascade  for $f \in \C$}, is  a sequence ${\mathcal A}_1, {\mathcal A}_2, \ldots,{\mathcal A}_n , \ldots $ of families of subintervals in $[0,1]$, satisfying the following properties:

 Each family ${\mathcal A}_n$ is composed by a finite number of pairwise disjoint compact intervals  with nonempty interiors, which are called \em atoms of generation $n$, \em  such that

\noindent $\bullet$ $\max_{\overline I \in {\mathcal A}_n} \mbox{length}(\overline I) \rightarrow 0$ as $n \rightarrow + \infty$.

\noindent $\bullet$ ${\mathcal A}_1 = \{\overline I_0, \overline I_1\}$  is a $2$-horseshoe for $f$, i.e.,
\begin{equation}
\label{eqn21}
\mbox{int}(f(\overline I_t)) \supsetneq \overline I_1 \cup \overline I_2  \mbox{ for } t= 0,1
\end{equation} (see Figure \ref{Figure1}). We have $\#{\mathcal A}_1 = 2$.

\begin{figure}[t]
\includegraphics[scale=.5]{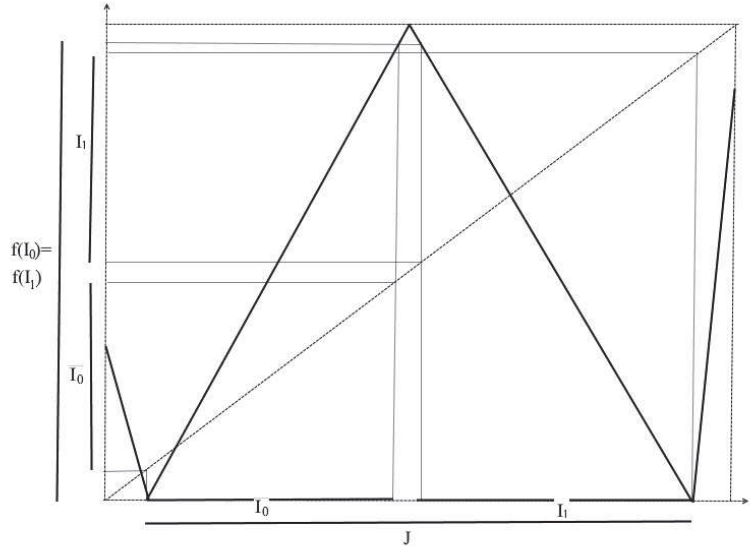}
\caption{\label{Figure1} A 2-horseshoe: The compact subintervals $\overline I_0$ and $\overline I_1$ are disjoint, and $f(I_0)\cap  f(I_1) \supset \overline I_0 \cup \overline I_1$.}
\end{figure}

\noindent $\bullet$ ${\mathcal A}_2 =   \{\overline I_{ [t]_2 }\colon [t]_2 \in T_2
\}$   such that
\begin{equation}
\label{eqn22}\overline I_{ [t]_2 } \subset I_{\pi_1 [t]_2}, \ \ \mbox{ and } \ \
 \mbox{int} ( f(I_{ [t]_2 } )) \supsetneq \overline I_{\sigma_1[t]_2}
 \end{equation}  (see  Figure \ref{Figure2} and Example \ref{ExampleCascade}).  Note that $\#{\mathcal A}_2 = \# T_2 = 2^3  $.

 \begin{figure}[h]
\includegraphics[scale=.5]{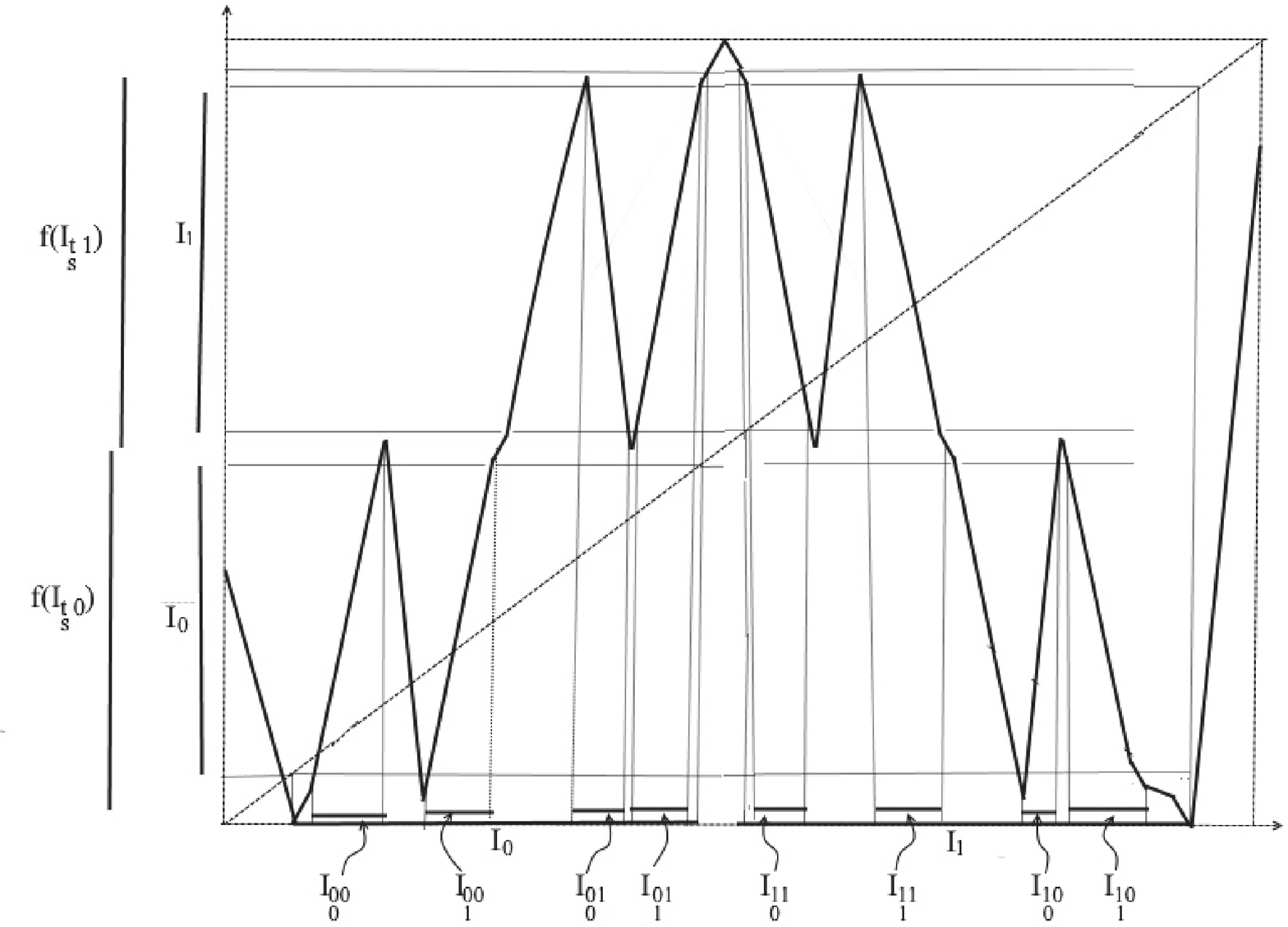}
\caption{\label{Figure2} An atom doubling cascade, up to generation 2. We modify the function of Figure \ref{Figure1} only inside the intervals $I_0$ and $I_1$. Each interval  $I_0$ and $I_1$ contains 4 pairwise disjoint compact subintervals $$ \overline I_ {   \underset{\mathbf{b}}{  0}  { \mathbf{ a}}     }, \ \  \overline I_ {   \underset{\mathbf{b}}{  1}  { \mathbf{ a}}     }$$ respectively, with $(a,b) \in \{0,1\}^2$. These subintervals are the atoms of generation 2. The map $f$ is now defined in these  subintervals such that
for all $(t,s) \in \{0,1\}^2$  $$f (I_ {   \underset{\mathbf{s}}{  \mathbf{t}}  {  { 0}}     } ) \supset \overline I_{0}, \ \ \ f (I_ {   \underset{\mathbf{s}}{  \mathbf{t}}  {  {1}}     } )\supset \overline I_1.$$}
\end{figure}

\noindent $\bullet$ ${\mathcal A}_3 =   \{\overline I_{ [t]_3 }\colon [t]_3 \in T_3
\}$   such that
\begin{equation}
\label{eqn23}\overline I_{ [t]_3 } \subset I_{\pi_2 [t]_3}\cap f^{-1} (I_{\pi_1 \sigma_2 [t]_3}), \ \ \mbox{ and } \ \
 \mbox{int} ( f(I_{ [t]_3 } )) \supsetneq \overline I_{\sigma_2[t]_3}.\end{equation}  (see  Figure \ref{Figure3} and Example \ref{ExampleCascade}).  Note that $\#{\mathcal A}_3 = \# T_3 = 2^6  $.

\noindent $\bullet$ In general, for all $n \geq 3$, ${\mathcal A}_n =   \{\overline I_{ [t]_n }\colon [t]_n \in T_n
\}$   such that
\begin{equation}
\label{eqn24}\overline I_{ [t]_n } \subset I_{\pi_{n-1} [t]_n}\cap f^{-1} (I_{\pi_{n-2} \sigma_{n-1} [t]_n}), \ \ \mbox{ and } \ \
 \mbox{int} ( f(I_{ [t]_n } )) \supsetneq \overline I_{\sigma_{n-1}[t]_n}.\end{equation}   So,  $\#{\mathcal A}_n = \# T_n = 2^{n(n+1)/2}  $.
\end{definition}

\begin{figure}[h]
\includegraphics[scale=.5]{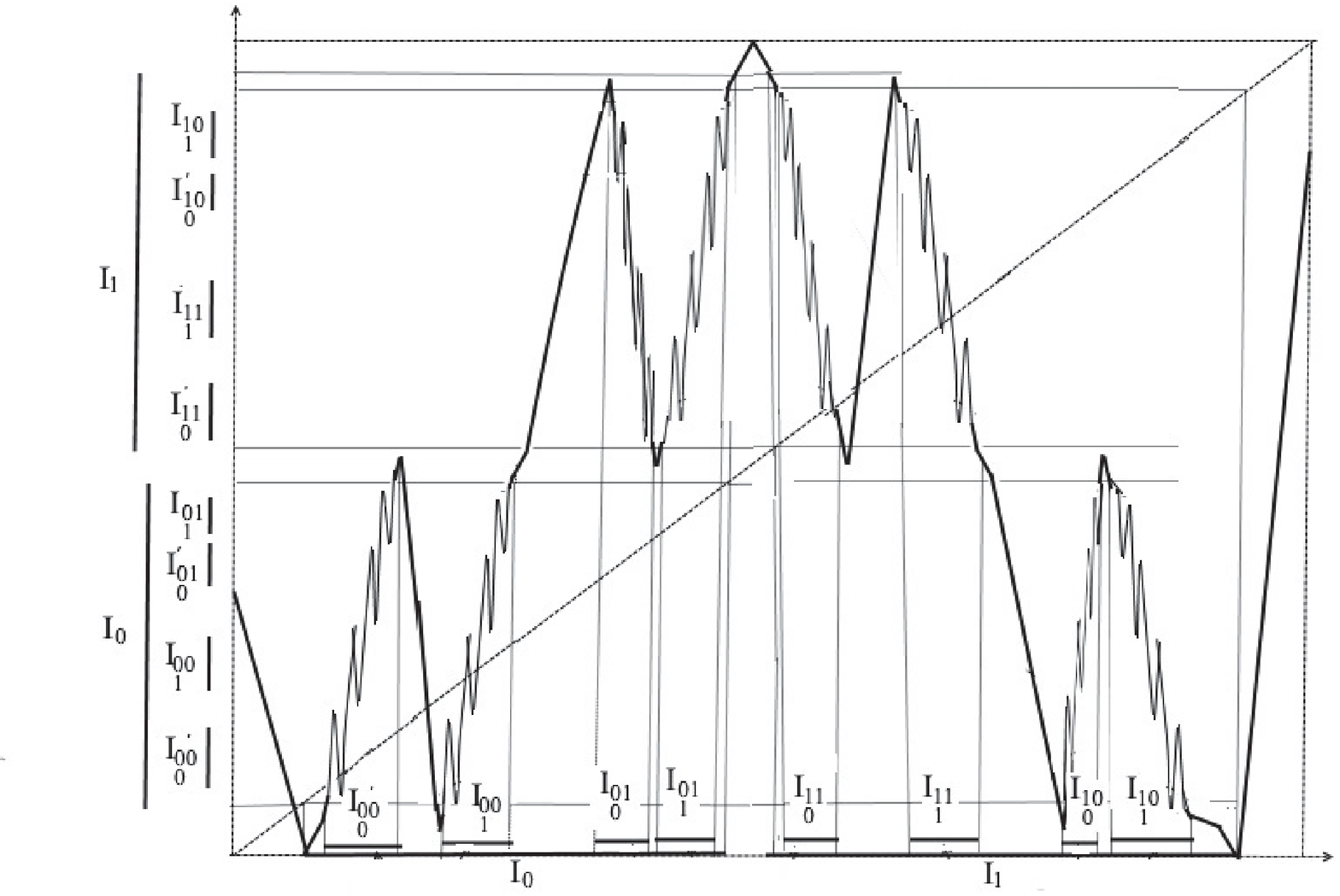}
\caption{\label{Figure3} An atom doubling cascade  up to generation 3. We modify the map $f$ of Figure \ref{Figure2} only inside the intervals $I_ {   \underset{  0}{  0}  {  0}     }, \ I_ {   \underset{  1}{  0}  {  0}}, \ I_ {   \underset{  0}{  0}  {  1}}, \ldots, I_ {   \underset{  1}{ 1}  { 1}}     $  (atoms of generation 2). Each of these intervals now contains 8 subintervals of level 3 (atoms of generation 3).  For instance $$I_ {   \underset{  0}{  0}  {  1}     } \supset   \overline  I_ { \underset{  \mathbf{c}}{ \underset{   0}{  0} } \underset{  \mathbf{ b}}{  1}  { \mathbf{ a}}   }   $$ for all  $  ( {a, b, c}) \in \{0,1\}^3  $. The map $f$ is now defined in these subintervals such that
$$I_{1} \supset  f \Big ( \overline{  I_ {  \underset{\mathbf{c}}{ \underset{0}{  0} } \underset{\mathbf{b}}{  1}  {\mathbf{a}} }  }  \Big)\supset f \Big (   I_ {  \underset{\mathbf{c}}{ \underset{0}{  0} } \underset{\mathbf{b}}{  1}  {\mathbf{a}}   }  \Big) \supset \overline I_ {   \underset{\mathbf{b}}{  1}  { \mathbf{ a}}     } $$
and such that
$$ f \Big (I_ { \underset{0}{ \underset{   0}{  0} } \underset{  \mathbf{ b}}{  1}  { \mathbf{ a}}   }  \Big ) = f \Big ( I_ { \underset{1}{ \underset{   0}{  0} } \underset{  \mathbf{ b}}{  1}  { \mathbf{ a}}} \Big ).$$}
\vspace{-.99cm}
\end{figure}

\begin{example} \em \label{ExampleCascade}
For any three open nonempty intervals $J, I, I'$ such that $  I' \supsetneq \overline I \supsetneq I \supset \overline J$, we will construct a map $f : \overline J \mapsto I' $ exhibiting an atom doubling cascade. The construction  illustrates Definition \ref{definitionCascade} and will be  useful at the end of this section to prove that typical maps in $\C$ exhibit such a cascade.

We will construct the map $f $ 
in several steps.
At  the $n$-th step, we will construct a continuous map $f_n: \overline J \mapsto I'$ that does not have a cascade, but has a finite family of atoms, up to generation $n$, satisfying conditions (\ref{eqn21})--(\ref{eqn24}).  We say that such map $f_n$ \emph{exhibits an atom doubling cascade  up to level $n$}. Finally, the map $f$  will be the uniform limit of the sequence of maps $f_n $.

\noindent{\bf Step 1 (construction of $f_1$).} We have $ \overline J \subset I$. Construct a continuous piecewise affine map $f_1: J \mapsto I$ as in Figure \ref{Figure1} (in this figure, the larger interval containing $\overline J$ is $I$; and the interval $I'$ is not drawn).   This map $f_1$ exhibits a $2$-horseshoe  $\{\overline I_1, \overline I_2\}$, with $\overline I_1, \overline I_2 \subset \overline J \subset I$. We construct $f_1$ so that  $|f_1'(x)| > 2$ for all $x \in   \overline I_1 \bigcup   \overline I_2$.

\noindent{\bf Step 2 (construction of $f_2$).} Take $f_1$ constructed in the first step. We construct $f_2$ such that $f_2(x) = f_1(x)$ if $x \not \in I_1 \bigcup I_2$, and modifying $f_2$ in $I_1  \bigcup I_2$ as in Figure \ref{Figure2}. The new map $f_2$  exhibits an atom doubling cascade $\{{\mathcal A}_1, {\mathcal A}_2\}$  up to generation 2. The family of atoms of generation 1 for $f_2$  and $f_1$ is the same ${\mathcal A}_1 = \{\overline I_0, \overline I_1\}$. Now, we have the family $${\mathcal A}_2 = \{I_{\subalign{&00\\&0}},I_{\subalign{&00\\&1}},I_{\subalign{&01\\&0}},\ldots,I_{\subalign{&11\\&1}}
     \},$$  of  atoms of generation 2, satisfying the conditions of   Definition \ref{definitionCascade}. The new map $f_2$ is continuous piecewise affine, and $$|f_2(x) - f_1(x)| < 2 \cdot \mbox{length}(I_i) <1   \ \ \forall \ x \in I_i, \ \ \mbox{ for } i= 0,1.$$ Besides, we construct $f_2$ such that $|f_2'(x)| > 2^2$ for all $x \in   \bigcup_{A \in {\mathcal A}_2} A$.

\noindent{\bf Step 3 (construction of $f_3$).} Take $f_2$ constructed in the second step. Let us construct $f_3$ such that $f_3(x) = f_2(x)$ if $x \not \in \bigcup_{A \in {\mathcal A_2}} A$, and modifying $f_2$  only in the interior of the atoms of generation 2,  as in Figure \ref{Figure3}. The new map $f_3$  exhibits an atom
doubling cascade $\{{\mathcal A}_1, {\mathcal A}_2, {\mathcal A}_3\}$  up to generation 3. The families ${\mathcal A}_1, {\mathcal A}_2$ of atoms of
generation 1 and 2, for $f_3$  and $f_2$, are the same.  Now, we have the family

\vspace{-.5cm}
$${\mathcal A}_3 \defeq \left \{I_{\subalign{&000\\&00\\&0}},
I_{\subalign{&000\\&00\\&1}},I_{\subalign{&000\\&01\\&0}},I_{\subalign{&001\\&00\\&0}},I_{\subalign{&000\\&01\\&1}},
 \ldots, I_{\subalign{&111\\&11\\&1}}  \right \},$$
  of  atoms of generation 3, satisfying the conditions of   Definition \ref{definitionCascade}. The new map $f_3$ is continuous piecewise affine, and $$|f_3(x) - f_2(x)| < 2 \cdot  {\mbox{length}(A)} < \frac{2}{2^2} = \frac{1}{2} \ \ \forall \ x \in A \in {\mathcal A}_2.$$~Besides, we construct $f_3$ such that $ |f_3'(x)| > 2^3 $ for all $ x \in   \bigcup_{A \in {\mathcal A}_3} A.$

\noindent{\bf Inductive step (construction of $f_{n+1}$ from $f_n$).} Assume that $f_n$ is constructed exhibiting an atom doubling cascade $\{{\mathcal A}_1, {\mathcal A}_2, \ldots, {\mathcal A}_n\}$  up to generation $n$. We construct $f_{n+1}$ such that $f_{n+1}(x) = f_n(x)$ if $x \not \in \bigcup_{A \in {\mathcal A_n}} A$, and modifying $f_n$  only in the interior of the atoms of generation $n$.  So, the families ${\mathcal A}_1, {\mathcal A}_2, \ldots {\mathcal A}_n$ of atoms of generation $1,2, \ldots, n$, for $f_{n+1}$  and $f_n$, are the same.  But we modify $f_n$ inside the atoms of generation $n$ so that, each of these atoms, say $\overline I_{[t]_n}$, contains  exactly $2^{n+1}$ pairwise disjoint atoms of generation $n+1$, say $\overline I_{[s]_{n+1}}$ such that $\pi_n [s]_{n+1} = [t]_n$. We require that the new family $${\mathcal A}_{n+1} = \left \{ \overline I_{[s]_{n+1}} \colon [s]_{n+1} \in T_{n+1} \right \},$$ of  atoms of generation $n+1$ satisfy the conditions of   Definition \ref{definitionCascade}. We construct the new map $f_{n+1}$ to be continuous piecewise affine,  and  $$|f_{n+1}(x) - f_n(x)| < 2 \cdot  {\mbox{length}(A)} < \frac{2}{2^n} = \frac{1}{2^{n-1}} \ \ \forall \ x \in A \in {\mathcal A}_n, $$ $$ |f_{n+1}'(x)| > 2^{n+1} \ \ \forall \ x \in   \bigcup_{A \in {\mathcal A}_n} A.$$
Note that this estimate on the derivative immediately implies that the length of the atoms goes to zero as $n$ tends to infinity.

\noindent{\bf Final Step (taking the limit map). }  We claim that the sequence of maps $f_n: \overline I \mapsto \overline I'$, as was constructed in the previous steps, uniformly converges to a continuous map $f :  \overline I \mapsto \overline I'$ as $n \rightarrow +\infty$.
By construction $$\sup_{x \in \overline I} |f_{n+1}(x)- f_n(x)| < \frac{1}{2^{n-1}} \ \ \forall \ n \geq 1 \ \ \mbox{ and thus}$$
$$\mbox{dist} (f_{n'}, f_n) < \sum_{j= n-1}^{\infty} \frac{1}{2^{j}} = \frac{1}{2^{n-2}} \ \ \forall \ n' \geq n \geq 1;$$
i.e., the sequence $\{f_n\}_{n \geq 1}$ is Cauchy,  hence it converges
to a continuous map $f : \overline I \mapsto \overline I'$.
By construct the map $f$ verifies \eqref{eqn21}-\eqref{eqn24}, thus it exhibits an atom doubling cascade.
\end{example}

\begin{definition} \em Let $f \in \C$ and assume that $f$ exhibits an atom doubling cascade $\{{\mathcal A}_n\}_{n \geq 1}$.
 We denote $$\Lambda_n \defeq \bigcup_{A \in {\mathcal A}_n} A \ \ \forall \ n \geq 1, \ \ \mbox{ and } \ \ \Lambda\defeq \bigcap_{n=1}^{\infty} \Lambda_n.$$ Since $\Lambda_n$ is compact and $\Lambda_{n+1} \subset \Lambda_n$ for all $n \geq 1$,  the  set  $\Lambda$ is nonempty and compact. It is called
  \emph {the $\Lambda$-set} of the cascade.  By construction it is invariant, and since the diameters of the atoms of generation $n$ go to zero as  $n \rightarrow + \infty$,  the $\Lambda$-set is a Cantor set.
\end{definition}

\begin{theorem} \label{theoremMeasureForCascades}
Suppose that $f \in \C$ has an atom doubling cascade, then there exists an $f$-invariant  ergodic measure $\mu$ supported on its $\Lambda$-set, such that $h_{\mu}(f) = + \infty$.
\end{theorem}
\begin{proof}
Consider the space of bi-infinite matrices $\Sigma \defeq \{0,1\}^{\mathbb{N}^+ \times \mathbb{N}^+}$.
For any $[t]\in  \Sigma$ we denote
$  [t] = \{t_{i,j}\}_{  i \geq 1, \ j \geq 1}.$
We  extend
the projection map $\pi_n: T_{n+1} \mapsto T_n$ defined at the beginning of this section to the space $\Sigma$ as follows:
 $$ \forall \ [t] = \{t_{i,j}\}_{i,j \geq 1} \in \Sigma, \ \ \pi_n [t] \defeq [t]_n =  \{t_{i,j}\}_{1\leq i \leq n, \ 1 \leq j + i \leq n+1} \in T_n,$$ where the triangular matrix $[t]_n$ with $n$ rows, is obtained from the infinite matrix $[t]$,   taking only its terms of the first $n$ rows, and in the $i$th row, taking only the terms in the first $n+1 -i$ columns.

We also extend the shift map  $\sigma_n: T_{n+1} \mapsto T_n$  to
 the (left) shift  $\sigma: \Sigma \mapsto \Sigma $ as follows:
  $$\forall \ [t] \in \Sigma, \ \ \sigma[t] \defeq [s] \mbox{ iff } \pi_n [s] = \sigma_n \pi_{n+1} [t]; $$ $$ \mbox{ i.e., }  s_{i,j} = t_{i+1,j} \ \ \forall \ 1 \leq i \leq n, \ \ \forall \ 1 \leq j \leq n+1-i, \ \ \forall \ n \geq 1.$$
  In other words, the shift $\sigma$ applied to an infinite matrix $[t]$, is the infinite matrix obtained from $[t]$ after erasing its first column.

We endow $\Sigma$ with the topology generated by the cylinders, and consider the corresponding Borel sigma-algebra in $\Sigma$. (A cylinder is a set in $\Sigma$ obtained by fixing the terms in a finite number of fixed positions of the infinite matrix   $[t]$.)

Now, we relate the space $\Sigma$ with the  $\Lambda$-set of the atom doubling cascade.
 Denote  by ${\mathcal A}_n$ the family of atoms of generation $n$ of the cascade.
For each point $x \in \Lambda$, define its \em itinerary \em $$h(x) \defeq  [t] \in \Sigma \mbox{ iff } \ x \in \overline I_{\pi_n [t]} \in {\mathcal A}_n \ \ \ \forall \ n  \in {\mathbb{N}^+}, \ \ \mbox{ (recall that }   \pi_n[t] \in T_n).$$

It is standard to check that $h: \Lambda \mapsto \Sigma$ is an homeomorphism that conjugates $f|_{\lambda}$, with the shift  $\sigma: \Sigma \mapsto \Sigma$,
i.e.,
$$h \circ f|_{\Lambda}  = \sigma \circ h.$$
Besides, for any atom $A_{[t]_n} \in {\mathcal A}_n$, where $[t]_n$ is a fixed matrix in $T_n$, we have: $$h(A_{[t]_n} \cap \Lambda) = B_{[t]_n} \defeq \{[s] \in \Sigma: \ \ \pi_n [s] = [t]_n\}.$$
Therefore, $B_ {[t]_n}$ is a cylinder that fixes the terms in  $n(n+1)/2$ positions of the matrix $[s]$. We call such a particular cylinder \em an atom of generation $n$ \em of the space $\Sigma$.


For any cylinder  $C_k $ defined by fixing  $k$ elements of the matrix  $[t] \in \Sigma $, let
$\nu(C_k) \defeq \frac{1}{2^k}$.
In particular, for any atom $B_ {[t]_n}$ of generation $n$, $\nu (B_ {[t]_n}) = 2^{-n(n+1)/2}$.
This equality defines  a pre-measure on the family of cylinders. Since the family of all the cylinders generates the Borel sigma-algebra of the space $\Sigma$, the
pre-measure $\nu $ extends to a unique Borel measure $\nu$ on $\Sigma$. Besides, $\nu$ is
$\sigma$-invariant, since the   pre-measure $\sigma$-invariant. To see this consider a cylinder $C_k$, obtained by fixing   $k$ fixed elements of the matrix $[t] \in \Sigma$. Then, $ \sigma^{-1}(C_k) $ is also a cylinder, obtained by fixing  $k$ fixed elements  each matrix $[s] \in \sigma^{-1}(C_k)$: the positions of the elements  fixed in  a matrix  $[s] \in \sigma^{-1}(C_k)$, are obtained by shifting to the right the $k$ positions  fixed in a matrix $[t] \in C_k$. Therefore
$$\nu(\sigma^{-1} (C_k)) = \frac{1}{2^k} = \nu(C_k) \ \ \ \forall  \ \mbox{cylinder} \ C_k. $$

Consider the pull-back measure, for any Borel set $A \subset [0,1]$:
 $$\mu (A) \defeq ({h^{-1}}^* \nu)(A \cap \Lambda) \defeq \nu (h (A \cap \Lambda)).$$
 The measure $\mu$ is well defined on the Borel sigma-algebra of $[0,1]$ because $h^{-1}: \Sigma \mapsto \Lambda$ exists and is continuous, hence measurable.
 Besides, since $h$ is a conjugation, it is an isomorphism of the measure spaces. Hence $\mu$ is $f$-invariant. Also, $\mu$ is ergodic for $f$ if and only if $\nu$ is ergodic for $\sigma$. Besides, $h_{\mu}(f) =h_{\nu}(\sigma) $. Thus, to finish the proof of Theorem \ref{theoremMeasureForCascades}, it is  enough to prove that $\nu$ is $\sigma$-ergodic,
 and that
  $  h_{\nu}(\sigma) = + \infty$.

To prove ergodicity, it is enough to prove  strong mixing; this holds because for any pair of cylinders $C_h$ and $C_k$, there exists $n_0$ such that $\nu(\sigma^{-n} (C_h) \cap C_k) = 1/2^{h+k} = \nu(C_h) \cdot \nu (C_k)$   for all $n \geq n_0$.

Finally, let us prove that   $h_{\nu}(\sigma) = + \infty$. By the definition of the metric entropy, we have $h_{\nu}(\sigma) = \sup_{\mathcal P} h({\mathcal P}, \nu)$, where the supremum is taken on all the finite measurable partitions ${\mathcal P}$ of $\Sigma$. So, it is enough to prove that  for any  $k \in \N^+$, the partition ${\mathcal P}_k$, which is composed by all the different atoms $B_{[t]_k}$ of generation $k$ in $\Sigma$,  satisfies \begin{equation}
\label{eqntobeproved}
h({\mathcal P}_k, \nu) = k \log 2 \ \ \ \forall \ k \geq 1.\end{equation}
Let us recall the definition of the entropy $h({\mathcal P}, \nu)$ of a partition ${\mathcal P}$:
$$h({\mathcal P}, \nu) \defeq \lim_{n \rightarrow + \infty} \frac{ H({\mathcal P}^n, \nu)}{n}, \mbox{ where } $$
$${\mathcal P}^n \defeq  \bigvee_{j= 0}^{n-1} \sigma^{-j}({\mathcal P}), \ \ \ \  H({\mathcal P}^n, \nu) \defeq -\sum_{B \in {\mathcal P}^n} \nu(B) \ \log \nu(B).$$

Consider the partition  ${\mathcal P}_k \defeq \{B_{[t]_k} = h(A_{[t]_k}) \in {\mathcal A}_k, \ \ [t]_k \in T_k\}$  of $\Sigma$ into atoms of generation $k$.  We have $\# {\mathcal P}_k = 2^{k(k+1)/2} $, since each $B \in {\mathcal P}_k$ is a different cylinder obtained by fixing the \lq\lq first\rq\rq \ $k(k+1)/2$ terms of the infinite matrices of $\Sigma$; precisely, they have fixed the elements in the positions $(i,j)$ with $1 \leq i \leq k, \ \ 1 \leq j \leq k+1-i$. Thus, $$\sigma^{-1} ({\mathcal P}_k ): =  \{ \sigma^{-1} (B) \colon B \in{\mathcal P}_k \} =  \{ \sigma^{-1}(B_{[t]_k}) \colon [t]_k \in T_k\}   $$
is the union of cylinders, which are obtained by fixing $t_{i,j}$ with $1 \leq i \leq k$, $2 \leq j \leq k+2-i.$ So,
$${\mathcal P}_k \vee \sigma^{-1} ({\mathcal P}_k ) \defeq  \{C \cap \sigma^{-1} (B) \colon C, B \in{\mathcal P}_k \}  $$ $$  =\{ B_{[s]_k}\cap \sigma^{-1}(B_{[t]_k}) \colon [s]_k,  [t]_k \in T_k\} $$
is also the union of cylinders, which (when nonempty) are obtained by fixing  the terms of each matrix in the positions $(i,j), \ \ 1 \leq i \leq k, \ \ 1 \leq j \leq k+2 -i$. The sets $B \in {\mathcal P}_k \vee \sigma^{-1} ({\mathcal P}_k )  $ \em are not atoms, \em each piece $B$ of this partition is the union of two atoms of generation $k+1$.

By induction on $n  $ it is standard to deduce that, for all $n \geq 2$
$$ {\mathcal P}_k ^n \defeq {\mathcal P}_k \vee \sigma^{-1} ({\mathcal P}_k )   \vee   \sigma^{-2} ({\mathcal P}_k ) \vee \ldots \vee   \sigma^{-(n-1)} ({\mathcal P}_k )  $$
 is also the union of cylinders, which (when nonempty) are obtained by fixing by fixing the terms of each matrix in the positions $(i,j), \ \ 1 \leq i \leq k, \ \ 1 \leq j \leq k+n -i$.
Once again,  the sets $B \in {\mathcal P}_k ^n  $  \em are   not atoms. \em   In fact, they have fixed elements firs the first column up to row $k$, but they do not have  fixed element in the $j$th row for  $j \geq k+1$. They are indeed the union of   a  finite number of atoms of generation $k + n-1$.

Now let us compute $H ({\mathcal P}_k ^n, \nu)$. As seen above, each piece $B  \in {\mathcal P}_k^n$ is a different cylinder obtained by fixing the terms in exactly $kn + (k(k-1)/2)$ positions of the matrix. Therefore
$$ \# {\mathcal P}_k^n = 2^{kn + (k(k-1)/2)}  \ \ \mbox{ and }$$
$$ \nu(B)= 2^{-kn + (k(k-1)/2)} \ \  \forall \ B \in {\mathcal P}_k^n.$$
From the above equalities we deduce
$$ H ({\mathcal P}_k ^n, \nu) \defeq - \sum_{B \in   {\mathcal P}_k ^n} \nu (B) \log \nu(B) = \Big( kn + \frac{k(k-1)}{2}\Big)  \ \log 2. $$
Therefore
$$ h({\mathcal P_k}, \nu) \defeq \lim_{n \rightarrow + \infty} \frac{ H({\mathcal P}_k^n, \nu)}{n} =
  \! \!  \lim_{n \rightarrow + \infty} \frac{   kn   +  (k(k-1)/  2)}{  n} \ \log 2 = k \ \log 2,$$
  proving  (\ref{eqntobeproved}) as wanted. This ends the proof of Theorem \ref{theoremMeasureForCascades}.
\end{proof}

\begin{theorem} \label{theorempseudo-physicalInfiniteEntropy}
For a typical map $f \in \C$

\noindent (i)  there exists an atom cascade, and

\noindent (ii)  there exists an   invariant ergodic  pseudo-physical measure  with infinite entropy, supported on the   $\Lambda$-set of an atom doubling cascade.
\end{theorem}

\begin{proof}
Assertion (ii) is a direct consequence of Assertion (i), Theorem \ref{theoremMeasureForCascades} and  Theorem \ref{TheoremSRB-l=ClosureErgodic}. To prove  Assertion (i), we   introduce the following notation.
We say that a map $f  \in \C$ \em belongs to the class ${\mathcal D} $ \em if it has an atom doubling cascade (see  Definition  \ref{definitionCascade}).

We claim that  the class ${\mathcal D} $ is a dense  $G_{\delta}$-set in $\C$.  First, let us see that ${\mathcal D} $ is a $G_{\delta}$ set in $\C$.
In fact,   ${\mathcal D} \neq \emptyset$ (see Example \ref{ExampleCascade}). So, we can take $f \in {\mathcal D}$. By the definition of cascade, there exists a sequence $\{{\mathcal A}_1(f), \ {\mathcal A}_2(f), \ \ldots, \ {\mathcal A}_n(f), \ \ldots\}$ of families of atoms of each generation $n \geq 1$ for $f$, satisfying the conditions of Definition \ref{definitionCascade}.

Fix  $N \geq 3$. Consider  the finite number of closed and pairwise disjoint intervals  in  $\bigcup_{n=1}^N {\mathcal A}_n(f)$ of the cascade for $f$. They satisfy conditions (\ref{eqn21})--(\ref{eqn24}) of the definition of cascade. These conditions are open. Precisely,  the \em same  finite family \em $\bigcup_{n=1}^N {\mathcal A}_n(f)$ of compact and pairwise disjoint intervals satisfy conditions (\ref{eqn21})--(\ref{eqn24}) for all $g \in \C$ close  enough to $f$, instead of $f$, say for all $g$ in the open neighborhood $U_N(f) \subset \C$. In brief, ${\mathcal A}_n(g) = {\mathcal A}_n(f) $ for all $g \in U_{N}(f)$.

Now, vary $N$  and  construct
 $G\defeq \bigcap_{N \geq 3, }^{+\infty} \bigcup_{f \in {\mathcal D}}U_{N}(f).$
Trivially $G \supset  {\mathcal D}$, and by construction $G \subset {\mathcal D}$. Therefore ${\mathcal D}$ is a $G_{\delta}$-set, as wanted.

Let us now prove that ${\mathcal D}$ is   a dense family   in $\C$.  Fix  $f \in \C$ and  $\e >0$. Let us construct $g \in {\mathcal D}$ such that $\mbox{dist}(g,f)< \e$.

Take $0<\delta <  \e/6 $ such that, if $\mbox{dist}(x,y) < \delta$, then $\mbox{dist}(f(x), f(y)) < \e/6$.  Consider a fixed point $x_0$ of $f$, and  the open interval $I \defeq (x_0- \delta, x_0 + \delta)  $. Construct  an open interval $J$, whose closure is contained in $I$, and a   map $g \in \C$ such that:

\noindent (i) $g(x) = f(x)$ if $x \not \in I;$

\noindent (ii) $g|_{\overline J}: \overline J \mapsto I' \defeq (x_0 - \e/6, \ \ x_0 + \e/6)$ is the map of   Example \ref{ExampleCascade}.

\noindent (iii) $g|_{\overline I \setminus J}$ is affine in each of the two connected components of $\overline I \setminus J$.
From  Condition (ii),  $g$ has an atom doubling cascade; namely $g \in {\mathcal D}$.

Let us check that $\mbox{dist}(g,f)< \e$. On the one hand, we have $|g(x)- f(x)|= 0$ if $x \not \in I$. On the other hand, since $f(x_0) = x_0$, if $ x \in \overline J$ then
\begin{equation*}
|g(x) - f(x)| < |g(x) - x_0| + |f(x_0) - f(x)| <\frac{\e}{6} + \frac{\e}{6} < \e
\end{equation*}

Finally, if $x \in \overline I \setminus J$, then there exists $y  \in \partial I$ and $z \in \partial J$  such that $x $ belongs to the interval with extremes $y,z$, and  $g$ is affine in this interval. Then, recalling that $g(y) = f(y)$ and that we have already proved that $|g(z) - f(z)|< \e/3$, we deduce:
\begin{equation*}
\begin{split}
|g(x) - f(x)| & \leq |g(x)- g(y) |+ |f(y) - f(x)|\\
& \leq |g(z) - g(y)| +|f(y) - f(x)|\\
&  \leq  |g(z) - f(z)| + |f(z) - g(y)| +|f(y) - f(x)|\\
&=   | g(z) - f(z)| + |f(z) - f(y)| +|f(y) - f(x)|\\
& <  \frac{\e}{3} + \frac{\e}{6} +\frac{\e}{6} < \e.
\end{split}
\end{equation*}
We have proved that for any map $f \in \C$ and any number $\e>0$, there exists a map $g \in {\mathcal D}$ such that $\mbox{dist}(f,g)< \e$. Hence the family ${\mathcal D}$ of maps is dense in $\C$, as wanted.
\end{proof}


\begin{thebibliography}{99}

\bibitem[AA]{AA} F.\ Abdenur and M.\ Andersson,
\emph{Ergodic theory of generic continuous maps}
Comm.\ Math.\ Phys.\ 318 (2013), no. 3, 831--855.

\bibitem[AHK]{AHK} E.\ Akin, M.\ Hurley, and J.A.\ Kennedy.
\emph{Dynamics of topologically generic homeomorphisms}
Mem.\ Amer.\ Math.\ Soc., 164(783): viii+130, 2003.

\bibitem[AP]{AP} V.S.\, Alpern, and S.\ Prasad, \emph{Typical dynamics of volume preserving homeomorphisms},
Volume 139 of Cambridge Tracts in Mathematics. Cambridge: Cambridge University Press, 2000


\bibitem[CE1]{CE1} E.\ Catsigeras and H.\ Enrich,
\emph{SRB-like measures for $C^0$ dynamics}
 Bull.\ Pol.\ Acad.\ Sci.\ Math. 59 (2011), no. 2, 151--164.

\bibitem[CE2]{CE2} E.\ Catsigeras and H.\ Enrich,
\emph{Equilibrium states and SRB-like measures of $C^1$-expanding maps of the circle}
Port.\ Math.\ 69 (2012), no. 3, 193--212.

\bibitem[CCE1]{CCE1} E.\ Catsigeras, M.\  Cerminara, and H.\ Enrich,
\emph{The Pesin entropy formula for $C^1$ diffeomorphisms with dominated splitting}
Ergodic Theory Dynam.\ Systems 35 (2015) 737--761.

\bibitem[CCE2]{CCE2}  E.\ Catsigeras, M.\  Cerminara, and H.\ Enrich,
\emph{Weak Pseudo-Physical Measures and Pesin's Entropy Formula for Anosov $C^1$ diffeomorphisms}
to appear in  Contemporary Mathematics.

\bibitem[H1]{H1}  M.\ Hurley,
\emph{On proofs of the $C^0$ general density theorem}
Proc.\  Amer.\  Math.\  Soc., 124 (1996) 1305--1309.

\bibitem[H2]{H2} M. Hurley,
\emph{Properties of attractors of generic homeomorphisms}
Ergodic Theory Dynam.\ Systems, 16 (1996)1297--1310.

\bibitem[KMOP]{Oprocha&als}P.\ Koscielniak, M.\ Mazur, P.\ Oprocha and P.\ Pilarczyk,  \emph{Shadowing is generic-a continuous map case} Discr.\ Contin.\ Dyn.\ Syst.\ 34 (2014), 3591--3609.

\bibitem[OU]{OU} J.C.\ Oxtoby, and S.M.\ Ulam, \emph{Measure-preserving homeomorphisms and metrical transitivity}
Ann.\ of Math. (2) 42 (1941) 874--920.


\end{thebibliography}
\end{document}